\documentclass[reqno, 11pt]{amsproc}
\usepackage{amsmath}
\usepackage{amsthm}
\usepackage{amssymb}
\usepackage{graphicx}
\usepackage{amsfonts, flexisym}
\usepackage{latexsym}
\usepackage{setspace}
\usepackage{hyperref}
\usepackage{cite}
\usepackage{multicol}
\usepackage[active]{srcltx}
\usepackage{mathrsfs}
\usepackage{xcolor}
\usepackage{tikz,amsmath,amsfonts}
\usetikzlibrary{calc,intersections}
\usetikzlibrary{decorations.pathreplacing}
\usepackage[font=small,labelfont=bf]{caption}
\graphicspath{ {images/} }

\usepackage{lineno}
\let\oldenumerate=\enumerate
	\def\enumerate{
	\oldenumerate
	\setlength{\itemsep}{5pt}
	}
\let\olditemize=\itemize
	\def\itemize{
	\olditemize
	\setlength{\itemsep}{5pt}
	}

\renewcommand{\Re}{\operatorname{Re}}
\renewcommand{\Im}{\operatorname{Im}}
\newcommand{\D}{\mathbb{D}}
\newcommand{\C}{\mathbb{C}}
\newcommand{\CC}{\mathsf{C}}
\newcommand{\Z}{\mathbb{Z}}
\renewcommand{\O}{\mathscr{O}}
\newcommand{\T}{\mathbb{T}}
\newcommand{\R}{\mathbb{R}}
\newcommand{\RO}{\mathsf{RO}}

\newcommand{\RR}{\mathsf{R}}

\newcommand{\prodk}{\prod_{k=0}^{\infty}}

\newcommand{\norm}[1]{\| \, #1 \,\|}

\renewcommand{\phi}{\varphi}
\renewcommand{\emptyset}{\varnothing}

\newcommand{\N}{\mathfrak{N}}

\renewcommand{\leq}{\leqslant}
\renewcommand{\geq}{\geqslant}

\allowdisplaybreaks

\numberwithin{equation}{section}
\theoremstyle{plain}

\newtheorem{Corollary}[equation]{Corollary}
\newtheorem*{Corollary*}{Corollary}
\newtheorem{Theorem}[equation]{Theorem}
\newtheorem*{Theorem*}{Theorem}
\newtheorem{Lemma}[equation]{Lemma}
\theoremstyle{definition}
\newtheorem{Definition}[equation]{Definition}

\newtheorem{Example}[equation]{Example}

\allowdisplaybreaks

\begin{document}
\bibliographystyle{amsplain}

\title{Real complex functions}

\author{Stephan Ramon Garcia}
\address{   Department of Mathematics\\
Pomona College\\
Claremont, California\\
91711 \\ USA}
\email{Stephan.Garcia@pomona.edu}
\thanks{First author partially supported by National Science Foundation Grant DMS-1265973.}

\author{Javad Mashreghi}
\address{D\'epartament de Mathematiques et de Statistique, Universit\'e Laval, Qu\'ebec, QC, G1K 7P4, Canada}
\email{javad.mashreghi@mat.ulaval.ca}
\thanks{Second author partially supported by NSREC}

\author{William T. Ross}
\address{   Department of Mathematics and Computer Science\\
University of Richmond\\
Richmond, Virginia\\
23173 \\ USA}
\email{wross@richmond.edu}

\begin{abstract}
We survey a few classes of analytic functions on the disk that have real boundary values almost everywhere on the unit circle. 
We explore some of their properties, various decompositions, and some connections these functions make to operator theory. 
\end{abstract}

\maketitle

\section{Introduction}

In this survey paper we explore certain classes of analytic functions on the open unit disk $\D$ that
have real non-tangential limiting values almost everywhere on the unit circle $\T$.  
These classes enjoy some remarkable analytic, algebraic, and structural properties that connect to various problems in operator theory. 
In particular, these classes can be used to describe the kernel of a Toeplitz operator on the Hardy space $H^2$; 
to give an alternate description of the pseudocontinuable functions on $H^2$ (alternatively the non-cyclic vectors for the backward shift operator); 
to define a class of unbounded symmetric Toeplitz operators on $H^2$; 
and to define an analogue of the classical Riesz projection operator for the Hardy spaces $H^p$ when $0 < p < 1$. 

Much of this material originates in the papers \cite{SCSSC, MR2186351, MR2021044}, which, in turn,
stem from seminal work of  Aleksandrov \cite{Alek-1981, MR643380} and Helson \cite{Helson, Helson2}.  We do, however, provide
many new examples and a few novel results not discussed in the works above.  We also endeavor to make this
survey accessible and thus include as many proofs as reasonable.  We hope the reader
will  be able to follow along
and eventually make their own contributions.

\section{Function spaces}

In this section we review a few definitions and basic results needed for this survey.
The details and proofs can be found in the well-known texts \cite{Duren, Garnett}.

\subsection{Lebesgue spaces} 
Let $\D$ denote the open unit disk in the complex plane $\C$ and
let $\T$ denote the unit circle.  We let $m$ denote normalized Lebesgue measure on $\T$ and for $0 < p < \infty$, we let $L^p := L^p(\T,dm)$ denote the space of Lebesgue measurable functions on $\T$ for which 
\begin{equation*}
\|f\|_{p} := \left(\int_{\T} |f|^p \,dm\right)^{1/p} < \infty.
\end{equation*}
When $0 < p < 1$, the preceding does not define a norm (the Triangle Inequality is violated) although $d(f, g) = \|f - g\|_{p}^{p}$ defines a translation invariant metric with respect to which $L^p$ is complete. 
When $1 \leq p < \infty$, the norm $\|\cdot\|_{p}$ induces a well-known Banach space structure on $L^p$.
When $p = 2$, $L^2$ is a Hilbert space equipped with the standard inner product
\begin{equation*}
\langle f, g \rangle = \int_{\T} f \overline{g} \,dm
\end{equation*}
and orthonormal basis $\{\zeta^{n}: n \in \Z\}$. 
When $p = \infty$, $L^{\infty}$ will denote the Banach algebra of essentially bounded functions on $\T$
endowed with the essential supremum norm $\|f\|_{\infty}$. 

If $f \in L^1$, then the function $\mathscr{P} f$ defined on $\D$ by 
\begin{equation}\label{eq:PIF}
(\mathscr{P} f)(z) = \int_{\T} f(\zeta) P_{z}(\zeta) \,dm(\zeta),
\end{equation}
denotes the {\em Poisson extension} of $f$ to $\D$,
where 
$$P_{z}(\zeta) = \Re\left(\frac{\zeta + z}{\zeta - z}\right), \quad \zeta \in \T, z \in \D.$$
 The function $\mathscr{P} f$ is harmonic on $\D$ and a theorem of Fatou says that
\begin{equation*}
\lim_{r \to 1^{-}} (\mathscr{P} f)(r \xi) = f(\xi) \quad \mbox{a.e. $\xi \in \T$}.
\end{equation*} 
We also let 
\begin{equation}\label{eq:QPIF}
(\mathscr{Q} f)(z) = \int_{\T} f(\zeta) Q_{z}(\zeta) \,dm(\zeta), 
\end{equation}
denote the {\em conjugate Poisson extension} of $f$,
where 
$$\qquad Q_{z}(\zeta) = \Im\left(\frac{\zeta + z}{\zeta - z}\right), \quad \zeta \in \T, z \in \D.$$
  The function $\mathscr{Q} f$ is also harmonic on $\D$ and 
\begin{equation}\label{ConPoi}
\widetilde{f}(\xi):=\lim_{r \to 1^{-}} (\mathscr{Q} f)(r \xi)
\end{equation} 
exists for almost every $\xi \in \T$, though the proof is more involved than for $\mathscr{P} f$.  The function $\mathscr{Q} f$ is the harmonic conjugate of $\mathscr{P} f$. One often thinks in terms of boundary functions and says that $\widetilde{f}$ is the harmonic conjugate of $f$. If $f$ has Fourier series
$$f \sim \sum_{n \in \Z} \widehat{f}(n) \zeta^n, \quad \widehat{f}(n) = \int_{\T} f(\zeta) \overline{\zeta}^{n} \,dm(\zeta),$$
then the conjugate function $\widetilde{f}$ has Fourier series 
\begin{equation}\label{sdiyfwr9efhudjk}
\widetilde{f}  \sim -i \sum_{n \not = 0} \mbox{sgn}(n) \widehat{f}(n) \zeta^n.
\end{equation}
A well-known theorem of M.~Riesz 
ensures that if $1 < p < \infty$ and $f \in L^p$ then $\widetilde{f} \in L^p$. This is known to fail when $p=1$ and $p=\infty$. 

\subsection{Hardy spaces}
For $0 < p < \infty$, the \emph{Hardy space} $H^p$ is the set of analytic functions $f$ on $\D$ for which
\begin{equation*}
\norm{f}_p:= \left(\sup_{0 < r < 1} \int_{\T} |f(r \zeta)|^p \,dm(\zeta) \right)^{1/p}< \infty.
\end{equation*}
When $1 \leq p < \infty$, $H^p$ is a separable Banach space while when $0 < p < 1$, $d(f, g) = \|f - g\|_{p}^{p}$ defines a translation invariant metric with respect to which $H^p$ is complete and separable.
We let $H^{\infty}$ denote the Banach algebra of all bounded analytic functions on $\D$ endowed with the norm 
$$\|f\|_{\infty} = \sup_{z \in \D}|f(z)|.$$  The Hardy spaces are nested in the sense that
$$H^q \subseteq H^p \iff p \leq q.$$

For each $f \in H^p$ with $0 < p < \infty$, the limit
\begin{equation*}
f(\xi):=\lim_{r \to 1^{-}} f(r \xi)
\end{equation*}
exists for almost every $\xi \in \T$ and 
\begin{equation}\label{eq:werwer}
\sup_{0 < r < 1} \int_{\T} |f(r \zeta)|^p \,dm(\zeta) = \lim_{r \to 1^{-}} \int_{\T} |f(r \zeta)|^p \,dm(\zeta) = \int_{\T} |f|^p \,dm.
\end{equation} 
Via these radial boundary values, $H^p$ can be identified with a closed subspace of $L^p$. In fact, for $p \geqslant 1$, 
\begin{equation*}
H^p = \Big\{f \in L^p: \widehat{f}(n) = 0, \,\, \forall n \leqslant -1\Big\}.
\end{equation*}

By a theorem of M. Riesz, the integral operator 
\begin{equation}\label{xcxcxrxrx}
f \mapsto  \int_{\T} \frac{f(\zeta)}{1 - \overline{\zeta} z} dm(\zeta), \quad f \in L^p,
\end{equation}
maps $L^p$ continuously onto $H^p$ when $1 < p < \infty$. In terms of Fourier series, this Riesz projection is equivalently defined by 
$$\sum_{n \in \Z} \widehat{f}(n) \zeta^n \mapsto \sum_{n \geq 0} \widehat{f}(n) \zeta^n.$$
The Riesz projection does not define a bounded operator from $L^1$ to $H^1$ nor a bounded operator from $L^{\infty}$ to $H^{\infty}$. We will revisit a version of this ``Riesz projection'' in \eqref{xcxcxrxrx} later on for subspaces of $L^p$ when $0 < p < 1$. 

\subsection{The canonical factorization} Each $f \in H^p$ can be factored as
$f = I F$, where $I \in H^{\infty}$ is an \emph{inner function}
and $F \in H^p$ is an \emph{outer function}. A general outer function is of the form 
$$F(z) = e^{i \gamma} \exp\left(\int_{\T} \frac{\zeta + z}{\zeta - z} \log w(\zeta) dm(\zeta)\right),$$ 
where $\gamma \in \R$,  $w \geqslant 0$, and $\log w \in L^1$. When $F \in H^p$ and outer, we have $\log |F| \in L^1$, $|F| \in L^p$,  and $w = \log |F|$. From \eqref{eq:PIF} and \eqref{ConPoi}, the radial boundary function  $F$ becomes 
\begin{equation}\label{bdryouter}
F = \exp(w + i \widetilde{w} + i \gamma) \quad \mbox{a.e.~on $\T$},
\end{equation}
which will be important later on. Specific examples of outer functions include: any zero free function that is analytic in a neighborhood of the closed unit disk; $f \in H^1$ with $\Re f > 0$;  and, in particular, functions of the form $1 + g$ where $g$ is an analytic function with $g(\D) \subseteq \D$.

The inner factor $I$ is a bounded analytic function on $\D$ with unimodular boundary values almost everywhere on $\T$ (the definition of an inner function) and can be factored further as 
\begin{equation}\label{eujdhdu998}
I = B S_{\mu}.
\end{equation} 
In the above,
$$B(z) = \xi z^{N} \prod_{n \geqslant 1} \frac{\overline{a_n}}{|a_n|} \frac{a_n - z}{1 - \overline{a_n} z},$$
is a {\em Blaschke product} with zeros at the origin (of multiplicity $N$) and at $\{a_n\}_{n \geqslant 1} \subseteq \D \setminus \{0\}$ (repeated according to multiplicity) with 
$$\sum_{n \geqslant 1} (1 - |a_n|) < \infty$$
(the {\em Blaschke condition}); $\xi$ is a unimodular constant; and 
$$S_{\mu}(z) = \eta \exp\left(-\int_{\T} \frac{\zeta + z}{\zeta - z} d\mu(\zeta)\right),$$
 called the {\em singular inner factor}, 
where $\mu$ is a finite positive measure on $\T$ with $\mu \perp m$ and $\eta$ is a unimodular constant. Up to unimodular constants, the factors in the canonical factorization are unique.

\subsection{Smirnov class} 

If $\mathscr{O}(\D)$ denotes the set of all analytic functions on $\D$, we define the following sub-classes of analytic functions: 
\begin{align*}
\mathfrak{N} & = \left\{\frac{f}{g}: f \in H^{\infty}, g \in H^{\infty} \setminus \{0\}\right\},\\
N &= \mathfrak{N} \cap \mathscr{O}(\D),\\
N^{+} & = \left\{\frac{f}{g}: f, g \in H^{\infty}, g \; \mbox{outer}\right\}.
\end{align*}
The class $\mathfrak{N}$ consists of the {\em meromorphic functions of bounded type}, $N$ is called the {\em Nevanlinna class}, and $N^{+}$ is called the {\em Smirnov class}. Note that 
$$N^{+} \subseteq N \subseteq \mathfrak{N}$$ and a standard theorem shows that 
$$\bigcup_{p > 0} H^p \subseteq N^{+}.$$
Functions in $\mathfrak{N}$ have radial boundary values a.e.~on $\T$. Furthermore, the radial boundary function of an element of $\frak{N} \setminus \{0\}$ is log-integrable.  As a consequence of the canonical factorization theorem, each $f \in \mathfrak{N} \setminus \{0\}$ can be written as 
\begin{equation}\label{eq:NFT}
f = \xi \frac{u_1}{u_2} F,
\end{equation}
where $\xi$ is a unimodular constant, $u_1, u_2$ are inner, and $F$ is outer. If $f \in N$, then $u_2$ is a singular inner function.
If $f \in N^{+}$, we simply have $f = I F,$ where I is inner and $F$ is outer. In particular, all inner functions and all outer functions belong to $N^{+}$.  
Focusing on $N^{+}$, the quantity 
\begin{equation}\label{eq:NMetric}
\rho(f, g) = \int_{\T} \log(1 + |f - g|) dm
\end{equation}
defines a metric on $N^{+}$ which makes $N^{+}$ a complete topological algebra. 

Before leaving this subsection, we want to state a theorem of Smirnov which says that 
\begin{equation}\label{SmirnovThm}
f \in N^{+} \; \;  \mbox{and} \; \;  f \in L^p \implies f \in H^p.
\end{equation}

\subsection{Classes of real complex functions}

We now arrive at the main focus of this survey: analytic functions on $\D$ that have
real boundary values a.e.~on $\T$. We introduce several classes of such functions and 
then focus on each particular class in a separate section. To get started, we define the following classes of functions:
\begin{align*}
\RR  &= \{f \in \mathfrak{N}: \mbox{$f$ is real valued a.e.\! on $\T$}\}, \\
\RR ^{+} &= \RR  \cap N^{+}, \\
\mathsf{RO} &= \{f \in \RR ^{+}: \mbox{$f$ is outer}\},\\
\RR ^{p} &= \RR ^{+} \cap H^p.
\end{align*}
The class $\RR$ is the {\em real Nevanlinna class}, $\RR^{+}$ the {\em real Smirnov class}, $\mathsf{RO}$ the {\em real outer functions}, and $\RR^{p}$  the {\em real $H^p$ functions}. We will give plenty of examples of these ``real'' functions below. As one of the simplest examples, consider 
$$f(z) = i \frac{1 + z}{1 - z} \in N^{+}.$$
By direct computation, one shows that whenever $\theta \not = \pi/2$ or $3 \pi/2$,
$$f(e^{i \theta}) = \cot\Big(\frac{\theta}{2}\Big) \in \R$$
and thus $f \in \RR^{+}$. In fact, $f \in \mathsf{RO} \cap \RR^p$ for all $0 < p < 1$. 

\section{Elementary observations}

There are a number elementary observations that can be made about real Smirnov functions.
Most of these involve standard results about Poisson integrals, linear fractional transformations, and inner-outer factorization in $N^+$.
These results, however simple, set the stage for the deeper results that are to follow.

\subsection{Helson's representation}\label{Subsection:Helson}

    The following theorem of Helson \cite{Helson2} (see also \cite{Helson}) provides a concrete description of several classes of real Smirnov functions.
    Unfortunately, it is difficult to use in practice since it involves sums and differences of inner functions.  For instance, it is difficult to determine the inner-outer factorization of a sum or difference of inner functions.

    \begin{Theorem}[Helson]\label{Theorem:Helson}
        Let $f \in \mathfrak{N}$.
        \begin{enumerate}
            \item If $f \in \RR$, then there are relatively prime inner functions $\psi_1$ and $\psi_2$ so that
              \begin{equation}\label{eq:Helson}
                    f =i \frac{\psi_1 + \psi_2}{\psi_1 -  \psi_2}.
              \end{equation}
             Up to a common unimodular constant factor, the inner functions $\psi_1$ and $\psi_2$ in \eqref{eq:Helson} are uniquely
             determined by $f$.  
            
             \item If $f \in \RR ^+$, 
             then there are relatively prime inner functions $\psi_1$ and $\psi_2$ so that 
             $\psi_1 -  \psi_2$ is outer and $f$ is of the form \eqref{eq:Helson}.
             
             \item If $f \in \mathsf{RO}$, 
             then there are relatively prime inner functions $\psi_1$ and $\psi_2$ so that 
             $\psi_1^2 - \psi_2^2$ is outer and $f$ is of the form \eqref{eq:Helson}.
         \end{enumerate}
    \end{Theorem}

    \begin{proof}
        (a) Suppose that $f \in \RR$.  Observe that the linear fractional transformation 
        $$z \mapsto \frac{z - i}{z + i}$$ maps $\R$ to $\T \setminus \{1\}$. 
         It follows that $$\frac{f-i}{f+i} \in \mathfrak{N}$$ and this function is
        unimodular a.e.~on $\T$.  Then \eqref{eq:NFT} ensures that
        there are relatively prime inner functions $\psi_1, \psi_2$, determined up 
        to a common unimodular constant factor, so that 
        \begin{equation*}
            \frac{f-i}{f+i} = \frac{\psi_2}{\psi_1}, \quad \psi_1 \not = \psi_2.
        \end{equation*}
        After a little algebra, we obtain \eqref{eq:Helson}.
          
        \noindent(b) Since $\RR ^+ \subseteq \RR$, (a) says that each $f \in \RR ^+$ enjoys
        a representation of the form \eqref{eq:Helson}, in which $\psi_1$ and $\psi_2$ are relatively prime inner functions.
        Suppose that $u$ is an inner factor of the denominator $\psi_1 - \psi_2$.  Then, since $f \in \RR ^+$,
        $u$ must also be an inner factor of the numerator $\psi_1+ \psi_2$. This means that $u$ is a common inner factor of both $\psi_1$ and $\psi_2$ (i.e., $u$ must be a unimodular constant).
        We conclude that the function $\psi_1 - \psi_2$ has no non-constant inner factor, and is thus outer.
        
        \noindent(c) Proceeding as in (b), we see that $\psi_1 + \psi_2$ is also outer.  Thus
        $$\psi_1^2 -\psi_2^2 = (\psi_1 +\psi_2)(\psi_1 -\psi_2)$$ is outer as well.
    \end{proof}
    
 Observe that the converses of statements (a), (b), and (c) trivially hold. 
 
\begin{Example}
Let 
\begin{equation*}
\psi_1 = \frac{z+ \frac{1}{2}}{1 + \frac{1}{2}z}
\quad\text{and}\quad
\psi_1 = \frac{z- \frac{1}{2}}{1 - \frac{1}{2}z}.
\end{equation*}
Then a short computation confirms that
\begin{equation}\label{eq:JavadFunction}
f = i \frac{ \psi_1 + \psi_2}{ \psi_1 - \psi_2} = \frac{3iz}{2 - 2z^2}.
\end{equation}
If $z = e^{i \theta}$, then
\begin{equation*}
f(e^{i\theta}) = -\frac{3}{4} \csc \theta,
\end{equation*}
so $f$ maps $\D$ onto the complement of the rays $(-\infty,-\frac{3}{4}]$ and $[\frac{3}{4}, +\infty)$.
This is illustrated in Figure \ref{Figure:Javad}.
\end{Example}

\begin{figure}
    \begin{tabular}{cc}
        \includegraphics[width=0.45\textwidth]{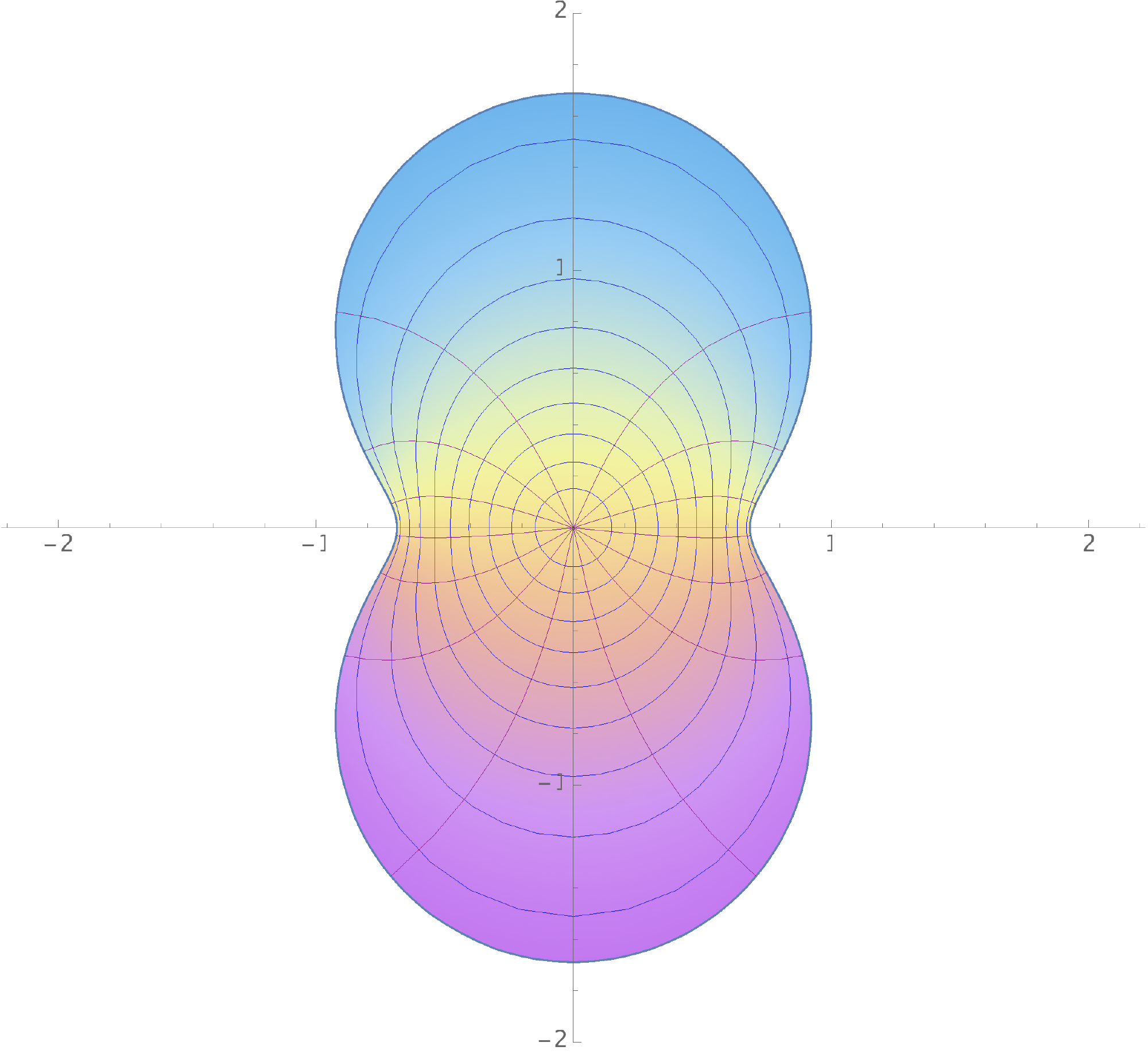} 
        & \includegraphics[width=0.45\textwidth]{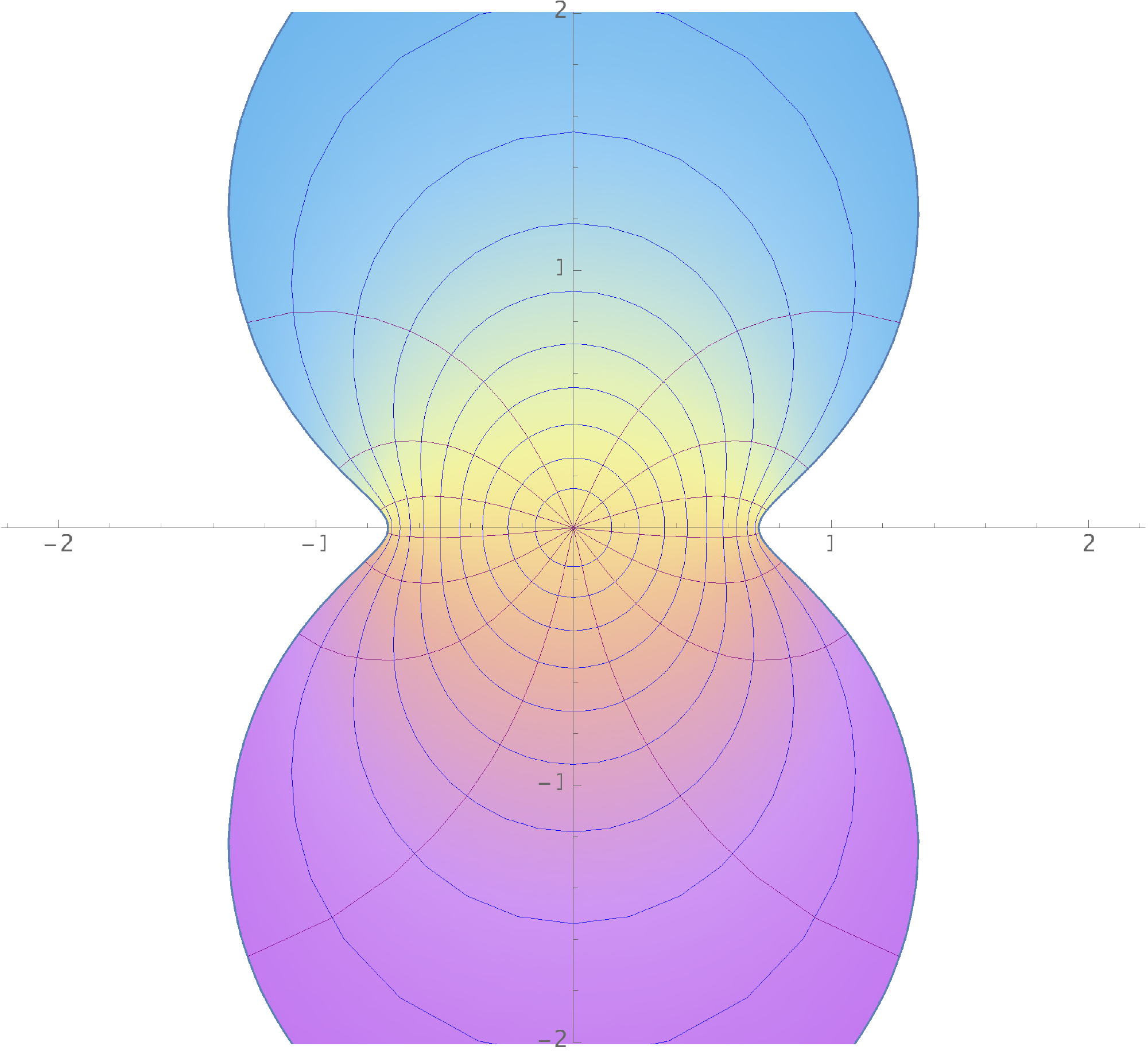} \\
        $r=0.65$ & $r = 0.75$\\[8pt]
         \includegraphics[width=0.45\textwidth]{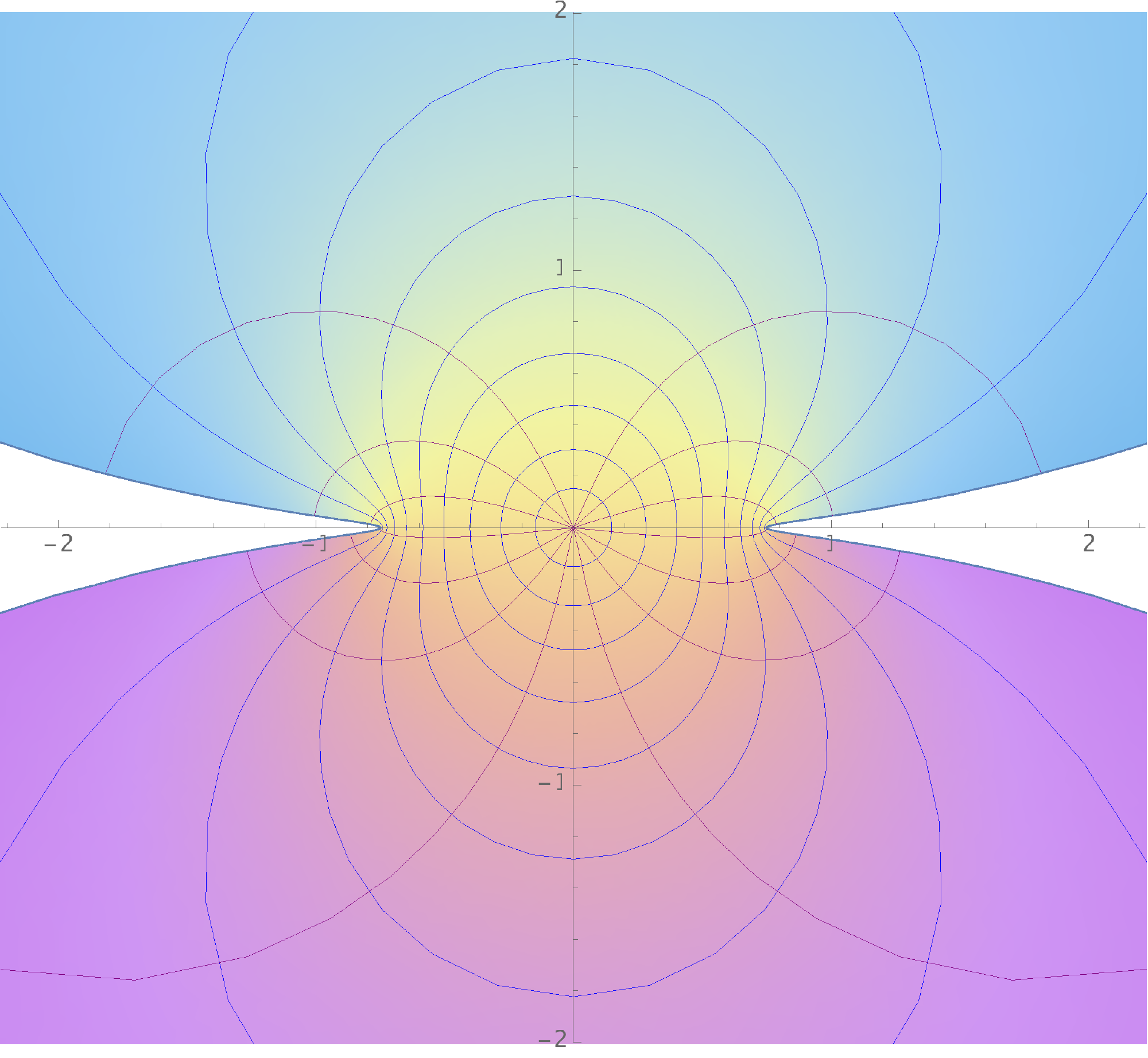} 
        &\includegraphics[width=0.45\textwidth]{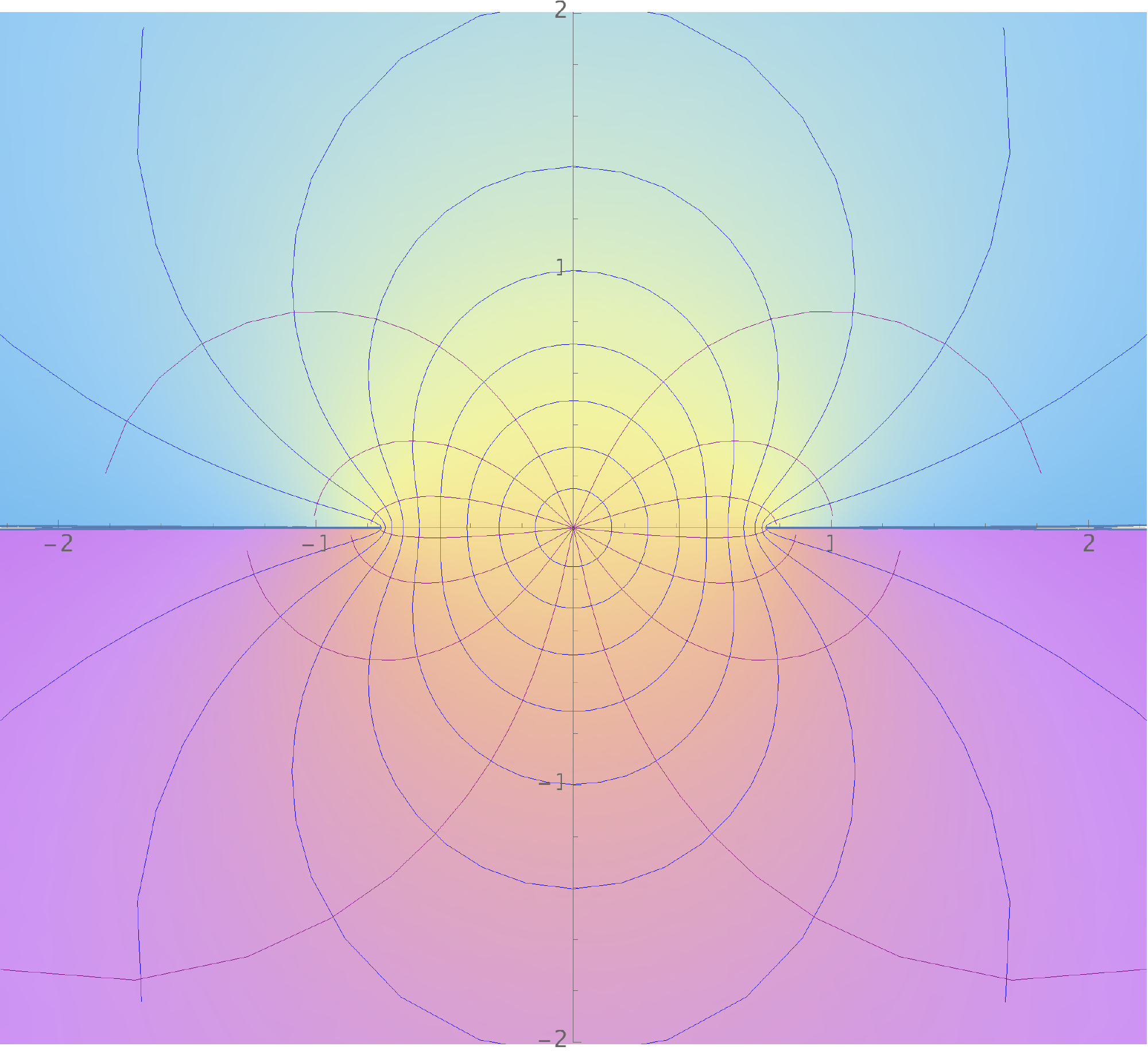} \\
        $r=0.95$ & $r = 0.999$
    \end{tabular}
    \caption{Images of the disk $|z| \leq r$ under the function \eqref{eq:JavadFunction} for several values of $r$.}
    \label{Figure:Javad}
\end{figure}

\subsection{Koebe inner functions}\label{Subsection:Koebe}    
    We ultimately seek to replace
    the Helson representation (Theorem \ref{Theorem:Helson}) with a more practical
    description of real Smirnov functions.  The first step is to reduce
    the consideration of functions in $\RR^+$ to the study of real outer functions (i.e., $\RO$).  To this end, we require the following definition.
    
    \begin{Definition}
        A \emph{Koebe inner function} is a function of the form
            $K(\phi)$,
        where $\phi$ is an inner function,
        $$K(z) = -4 k(z),$$ and
        \begin{equation}\label{eq:Koebe}
            k(z) = \frac{z}{(1 - z)^2}
        \end{equation}
        is the \emph{Koebe function}.
    \end{Definition}
    
    Recall that the Koebe function is a univalent map from $\D$ onto the complement of the half line
    $(-\infty,-\frac{1}{4}]$ in $\C$ \cite{DurenUnivalent}.  Thus $K$ maps $\D$ onto the 
    complement of the half line $[1,\infty)$ in $\C$.  In particular, $K \geq 1$
    a.e.~on $\T$ (Figures \ref{Figure:Koebe} and \ref{Figure:KoebeSingular}).  The following theorem provides
    an analogue of the canonical inner-outer factorization that is more suitable for real Smirnov functions.

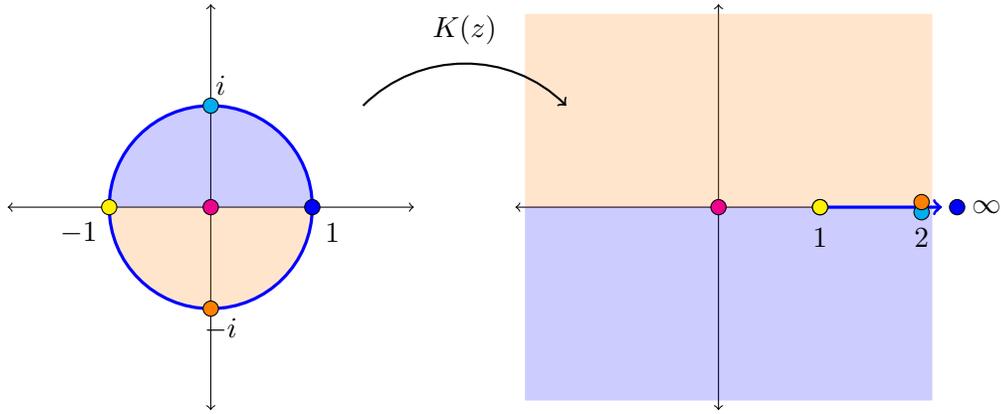
\begin{figure}
\centering
\begin{tikzpicture}[scale=1.35]
    \begin{scope}[xshift = -2.5cm]
        \draw[<->](-2,0)--(2,0);
        \draw[<->](0,-2)--(0,2);
    
        \draw[blue,fill=blue,opacity=.2](1,0)arc(0:180:1cm);
        \draw[orange,fill=orange,opacity=.2](1,0)arc(0:-180:1cm);
        \draw[very thick,blue](0,0) circle (1cm);
        \draw[black,fill=magenta] (0,0) circle (.075cm);
        \draw[black,fill=blue](1,0) circle (.075cm);
        \draw[black,fill=yellow] (-1,0)circle (.075cm);
        \draw[black,fill=cyan] (0,1) circle (.075cm);
        \draw[black,fill=orange] (0,-1) circle (.075cm);
        \node at (.1,1.2){$i$};
        \node at (.1,-1.2){$-i$};
        \node at (1.2,-.25){$1$};
        \node at (-1.3,-.25){$-1$};
    \end{scope}
    
    \node at (0,1.75) {$K(z)$};
    \begin{scope}[xshift = 2.5cm]
        \draw[<->](-2,0)--(2.2,0);
        \draw[<->](0,-2)--(0,2);
        \draw[orange,fill=orange,opacity=.2](-1.9,0) rectangle (2.1,1.9);
        \draw[blue,fill=blue,opacity=.2](-1.9,0)rectangle (2.1,-1.9);

        \draw[black,fill=magenta] (0,0) circle (.075cm);
        \draw[very thick,blue,->](1,0)--(2.2,0);

        \draw[black,fill=blue](1,0) circle (.075cm);
        \draw[black,fill=yellow] (1,0)circle (.075cm);
        \draw[black,fill=cyan] (2,-0.05) circle (.075cm);
        \draw[black,fill=orange] (2,0.05) circle (.075cm);
     \draw[black,fill=blue] (2.35,0)node[right]{\color{black}\,$\infty$} circle (.075cm);
       
        \node at (1,-.3){$1$};
        \node at (2,-.3){$2$};
    \end{scope}
    
    \draw[thick,->](-1,1)to[out=45,in = 135] (1,1);
\end{tikzpicture}
\caption{Illustration of the function $K(z) = -4k(z)$.  It is a univalent map from $\D$
onto the complement of the half line $[1,\infty)$ on the real axis.}
\label{Figure:Koebe}
\end{figure}

\begin{figure}
    \begin{tabular}{cc}
        \includegraphics[width=0.45\textwidth]{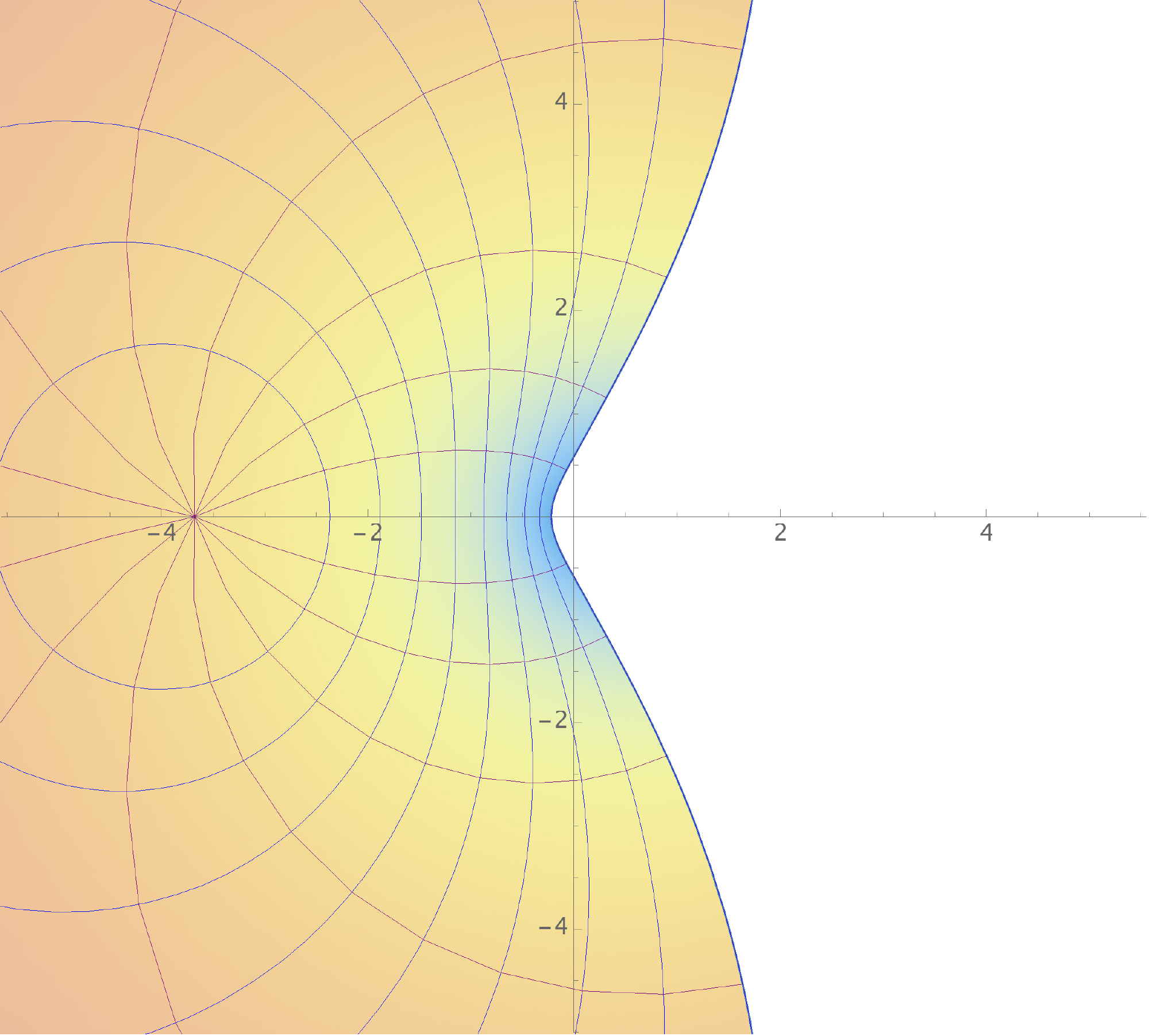} 
        & \includegraphics[width=0.45\textwidth]{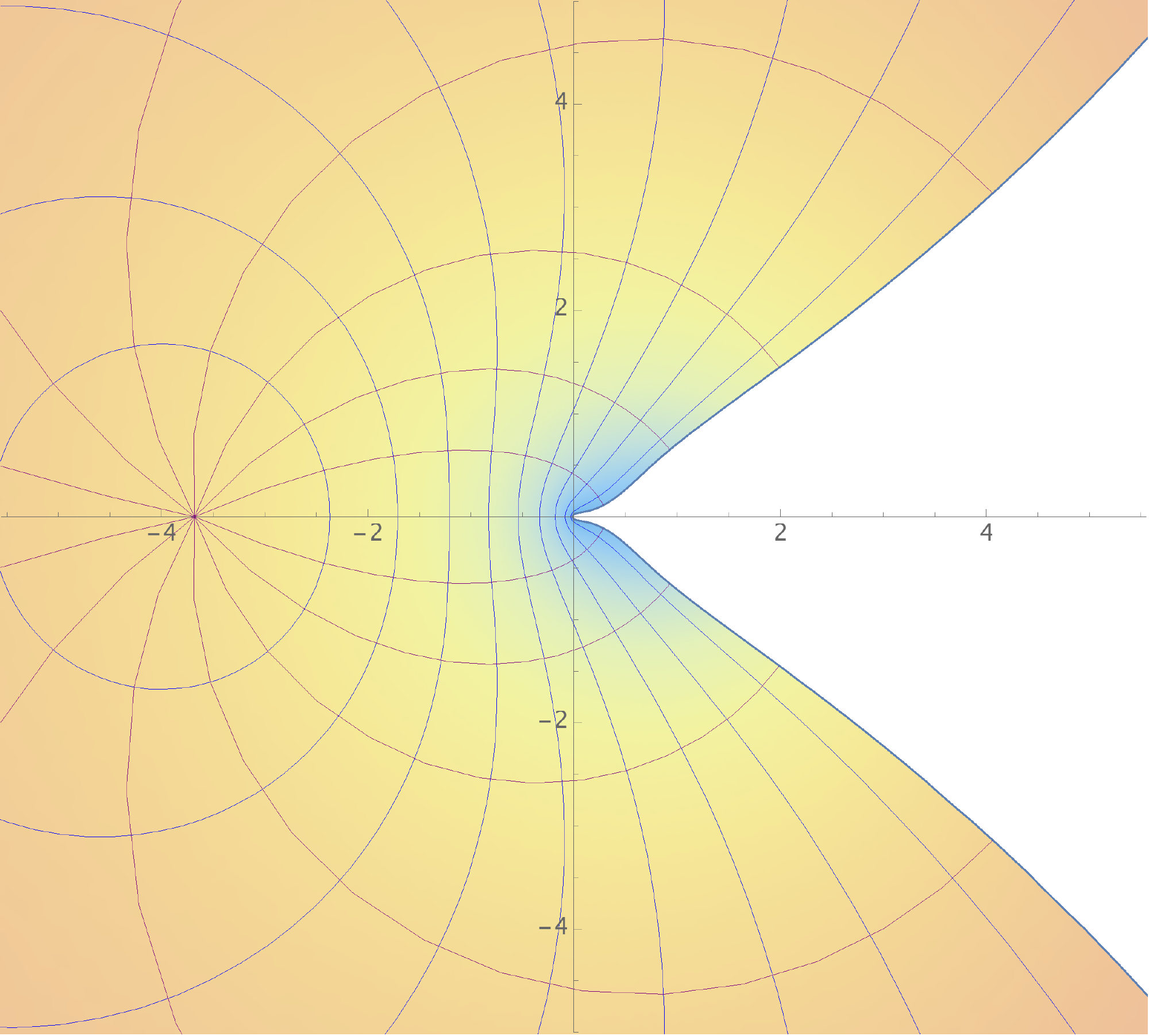} \\
        $r=0.5$ & $r = 0.7$\\[8pt]
         \includegraphics[width=0.45\textwidth]{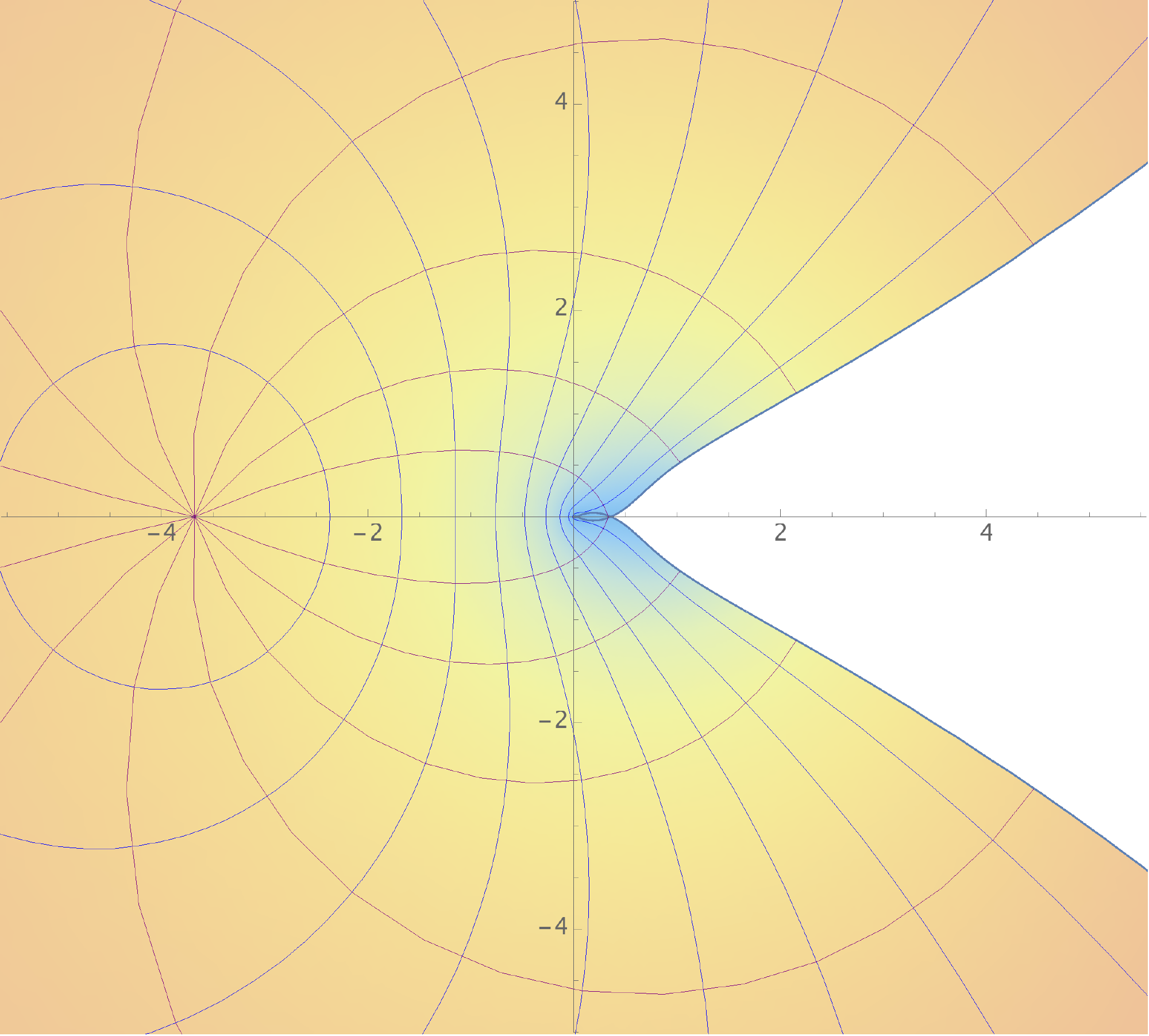} 
        &\includegraphics[width=0.45\textwidth]{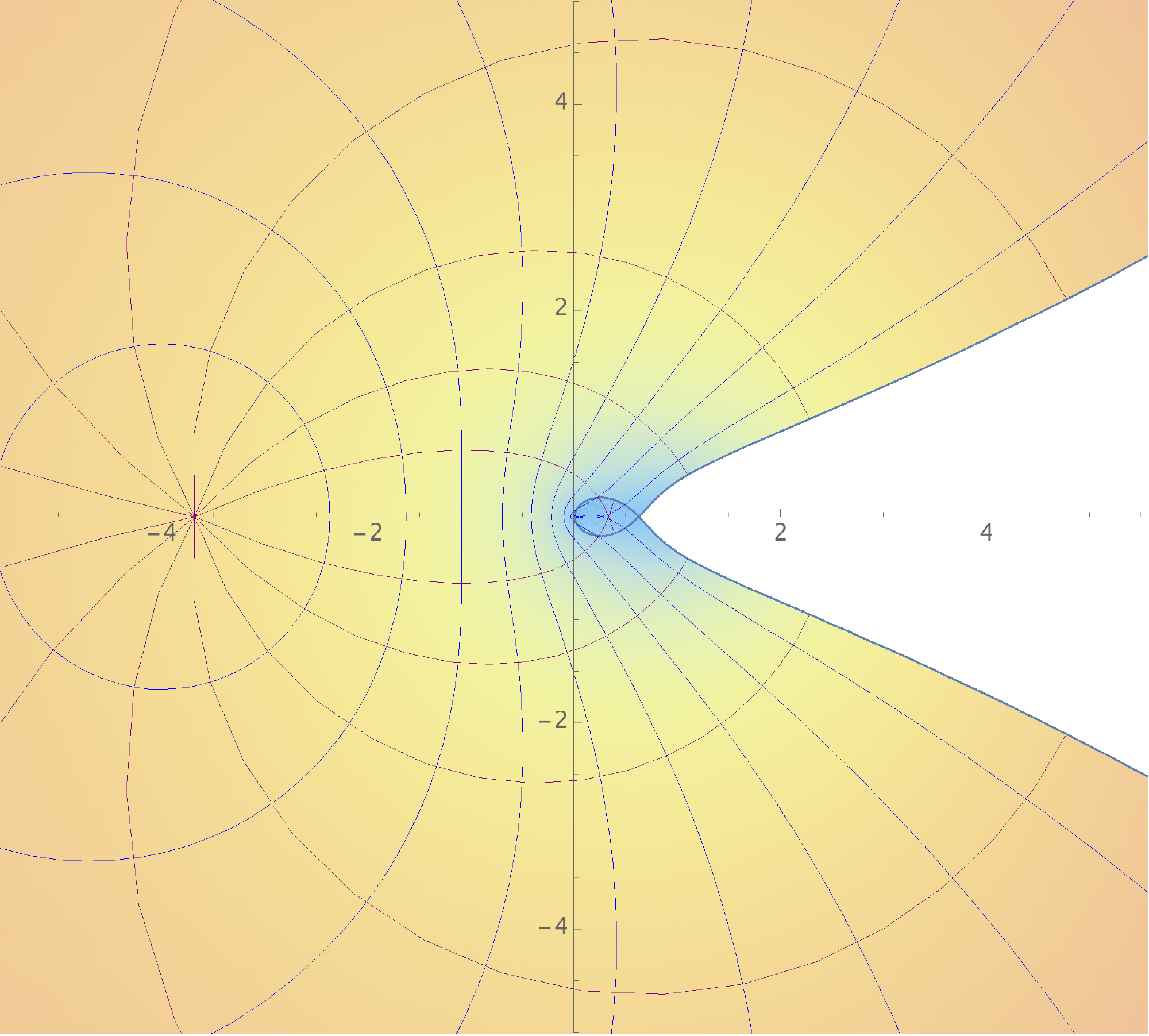} \\
        $r=0.75$ & $r = 0.800$
    \end{tabular}
    \caption{Images of the disk $|z| \leq r$ under the Koebe inner function function $K(\phi)$,
    where $\phi(z) = \exp[ (z+1)/(z-1)]$, for several values of $r \in (0, 1)$.  The image of $\D$ under $K(\phi)$
    is the complement of $\{0\}\cup [1,\infty)$.  This sequence of images suggests how $K(\phi)$ ``wraps'' $\D$ around $0$
    with infinite multiplicity.  The function is univalent on $|z| \leq r$ for small $r$; as $r \to 1^-$, the function $K(\phi)$
    pulls $\D$ through the gap $(0,1)$ and begins wrapping it around $0$.}
    \label{Figure:KoebeSingular}
\end{figure}

    \begin{Theorem}\label{Theorem:Factor}
        If $f \in \RR^+$ (resp., $\RR^p$), then 
        $f = K_f R_f$, where $K_f$ is Koebe inner and $R_f \in \RO$ (resp., $\RR^p$).  Moreover,
        \begin{enumerate}
            \item $|R_f| \leq |f|$ a.e.~on $\T$;
            \item $f$ and $R_f$ have the same sign a.e.~on $\T$.
        \end{enumerate}
    \end{Theorem}
    
    \begin{proof}
        Let $f \in \RR^+$ (resp., $\RR^p$) with inner factor $I_f$ and outer factor $F$.  Without loss of generality, we assume that $I_f \neq 1$. Otherwise, we may take $I_f=i$ and replace $F$ by $-iF$ (see the comment after the proof). Then
        \begin{equation*}
            f = \frac{-4 I_f}{(1 - I_f)^2} \cdot \frac{(1 - I_f)^2}{-4} F = K_f R_f,
        \end{equation*}
        where $K_f = -4k(I_f)$ is Koebe inner and $$R_f = -\frac{1}{4}(1 - I_f)^2 F$$ is outer.  
        Since $K_f \geq 1$ a.e.~on $\T$, the outer function $R_f = f/K_f$ is real a.e.~on $\T$, so it
        belongs to $\RO$.  Moreover, $|R_f| = |f/K_f|  \leq |f|$
        a.e.~on $\T$, so $R_f \in \RR^p$ whenever $f \in H^p$.  Since $f/R_f = K_f \geq 1$ a.e.~on $\T$,
        the functions $f$ and $R_f$ have the same sign a.e.~on $\T$.
    \end{proof}
    
    In what sense is the factorization $f = K_f R_f$ in Theorem \ref{Theorem:Factor} unique?
    Observe that the inner factor of $f$ is hidden in $K_{f}$ (they have the same inner factor up to unimodular constants). Modulo this constant, $K_{f}$, and hence $R_f$, is unique. 

\subsection{Growth restrictions}\label{Subsection:Growth}    
    
    The following theorem tells us that $\RR ^p$ is of interest only when $0 < p < 1$. 
    
    \begin{Theorem}\label{Theorem:p1}
        If $p \geqslant 1$, then $\RR ^{p} = \R$.  
    \end{Theorem}
    
    \begin{proof}
        When $1 \leqslant p < \infty$, we have $\RR ^p \subseteq \RR ^1$ and so it suffices to prove $\RR ^1 = \R$. 
        If $f \in \RR^1$ then $f \in H^1$ and so $g = if$ can be recovered from the
        analytic completion of its Poisson integral \cite[Thm.~3.1 \& p.~4]{Duren}, i.e., 
        \begin{equation*}
            g(z) = i\gamma + \int_{\T} \frac{\zeta + z}{\zeta - z} \Re g(\zeta)\,dm(\zeta)
        \end{equation*}
        for some real constant $\gamma$.
        The integral in the above expression is identically zero since $\Re g = 0$. Thus $f = \gamma$ is a constant function.
    \end{proof}
    
\begin{Example}
There are many examples of functions in $\O(\D)$
that have real non-tangential boundary values a.e.~on $\T$, but which
do not belong to $\N$ (that is, they are not of bounded type).
Indeed, if $f \in \RR^+$ (and non-constant), then  $e^f$ also has real (non-tangential) boundary values a.e.~on $\T$.  
However,  $e^f \not \in \RR^+$ since otherwise, Riesz's theorem (log integrability of Nevanlinna functions on $\T$) would imply that
$$\log|e^f| = \Re f  = f \in L^1$$ and hence, by Smirnov's theorem \eqref{SmirnovThm}, $f \in H^1$. 
Theorem \ref{Theorem:p1} now ensures that $f$ is constant, a contradiction.
As an amusing consequence, we note that if $f \in \RR^+$ is non-constant,
then 
$$\frac{e^f - i}{e^f + i}$$ is meromorphic on $\D$, unimodular a.e.~on $\T$, but
\emph{not} expressible as the quotient of two inner functions.
\end{Example}

    As mentioned earlier, if $\phi$ is inner, then $1 - \phi$ is outer.
    As a result, the inner factor of $K(\phi)$ is $\phi$.  
    Since the function $z \mapsto (1-z)^{-1}$ belongs to $H^p$ for $0 < p < 1$ \cite[p. 13]{Duren},
    the Littlewood Subordination Principle \cite[p.10]{Duren} ensures
    that
    \begin{equation}\label{cvbnmdfdfi} 
    (1- \phi)^{-1} \in H^p \quad \forall p \in (0, 1).
    \end{equation} Thus, a non-constant Koebe inner function
    belongs to $\RR^p$ for all $0 < p < \frac{1}{2}$.  This exponent is sharp, as we will see in a moment. The following result is due to Helson and Sarason \cite{MR0236989}
    and to Neuwirth and Newman \cite{MR0213576}.
    
    \begin{Theorem}\label{Theorem:HelsonSarason}
        If $f \in \RR ^{\frac{1}{2}}$ and $f \geqslant 0$ a.e.~on $\T$, then $f$ is constant.
    \end{Theorem}
    
    \begin{proof}
        Suppose that $f \in \RR ^{\frac{1}{2}}$ and $f \geqslant 0$ a.e.~on $\T$.
        Theorem \ref{Theorem:Factor} says that $f = K_f R_f$, in which $K_f$ is Koebe inner
        and $R_f \in \RR^{\frac{1}{2}}$ is outer and non-negative a.e.~on $\T$.
        Then $R_f^{\frac{1}{2}}$ is outer and belongs to $\RR^1$, so it is constant (Theorem \ref{Theorem:p1}).  Consequently, we may assume that $f = K(\phi) \in \RR^{\frac{1}{2}}$ is a Koebe inner
        function and thus is non-negative a.e.~on $\T$.  Observe that 
        \begin{equation*}
            \left(i \frac{1+\phi}{1-\phi} \right)^2   =  \frac{-4\phi - (1-\phi)^2}{(1-\phi)^2} =  K(\phi) - 1 \in \RR^{\frac{1}{2}}
        \end{equation*}
        is non-negative a.e.~on $\T$.  Consequently, $i(1+\phi)/(1-\phi)$ belongs to $\RR^1$
        so it is constant (Theorem \ref{Theorem:p1}).  Thus $\phi$ is constant as well. 
    \end{proof}
    
    The proof above yields this interesting corollary.
    
    \begin{Corollary}
    Any Koebe inner function belonging to $\RR^{\frac{1}{2}}$ must be constant. 
    \end{Corollary}
    
    If $f \in \RR^+$ is non-negative a.e.~on $\T$, it is not necessarily the case that $f$ is the square of a function
    in $\RR^+$.  For instance, $-4$ times the Koebe function \eqref{eq:Koebe} is a counterexample: It has a zero of order $1$ at $z= 0$ and hence cannot be the square of any analytic function.
    On the other hand, the following theorem asserts that a real Smirnov function that is non-negative a.e.~on $\T$
    is the sum of {\em two} squares of real outer functions.
    
    \begin{Theorem}
        If $f \in \RR ^{+}$ (resp., $\RR^p$) and $f \geqslant 0$ a.e.~on $\T$, then 
        $f = g_{1}^{2} + g_{2}^{2}$, where $g_1, g_2 \in \RO$ (resp., $\RO \cap H^{\frac{p}{2}}$). 
    \end{Theorem}
    
    \begin{proof} 
        Suppose that $f = I_f F \in \RR^+$ and $f \geqslant 0$ a.e.~on $\T$, where $I_f$ is inner and $F$ is outer.  Then $f$ can be written in two different ways as 
        \begin{equation*}
            f =\underbrace{ \frac{4I_f}{(1 + I_f)^2} }_{K(-I_f)}\cdot \frac{(1 + I_f)^2 F}{4} 
            = \underbrace{ \frac{-4I_f}{(1 - I_f)^2} }_{K(I_f)}\cdot \frac{(1 - I_f)^2 F}{-4}.
        \end{equation*}
       Note that $K(-I_f)$ and $K(I_f)$ are non-negative a.e.~on $\T$.
         It follows that $$\frac{1}{4}(1 + I_f)^2 F \quad \mbox{and} \quad -\frac{1}{4}(1- I_f)^2 F$$
        are both outer and non-negative a.e.~on $\T$.  Moreover, by direct verification, 
        \begin{equation*}
            \big[\tfrac{1}{4}(1 + I_f)^2 F \big]+ \big[- \tfrac{1}{4} (1 - I_f)^2 F\big] = f.
        \end{equation*}
        This expresses $f$ as the sum of two real outer functions that are non-negative a.e.~on $\T$.
        If $f \in H^p$, then $F$ is as well. Thus the two summands above also belong to $H^p$.
        Consequently, 
        $$g_1 = \frac{1}{2}(1 + I_f)\sqrt{F} \quad \mbox{and} \quad g_2 = \frac{i}{2}(1 - I_f)\sqrt{F}$$ are real outer functions
        that satisfy $f = g_1^2 + g_2^2$. They belong to $H^{\frac{p}{2}}$ whenever $f \in H^p$.
    \end{proof}

\subsection{Cayley Inner Functions}\label{Subsection:Cayley}

    Theorem \ref{Theorem:Factor} reduces the study of real Smirnov functions
    to the study of real outer functions.  The simplest real outer functions are, in essence,
    just Cayley transforms of inner functions.  For reasons that will become clear much later, we actually 
    require a certain variant of the Cayley transform that turns out to be
    compatible with infinite products in a crucial way.

    Consider the linear fractional transformation
    \begin{equation}\label{eq:T}
        T(z) = i \frac{1 - iz}{1 + iz},
    \end{equation}
    whose inverse is
    \begin{equation}\label{eq:T'}
        T^{-1}(z) = i \frac{z - i}{z + i}.
    \end{equation}
    One can verify that $T$ satisfies the identities 
    \begin{equation}\label{eq:Tp1}
        T^{-1}(z) = T(1/z) = \frac{1}{T(z)} = -T(-z) = \overline{T(\overline{z})}
    \end{equation}
    and
    \begin{equation}\label{eq:Tp2}
        (T \circ T)(z) = \frac{1}{z} \quad \text{and} \quad (T \circ T \circ T \circ T)(z) = z.
    \end{equation}
    The mapping properties of $T$ and $T^{-1}$ are illustrated in Figure \ref{Figure:T}
    and Table \ref{Table:T} below.
    \begin{table}[h!]
        \begin{equation*}
          \begin{array}{|c||c|c|c|c|c|c|c|c|c|c|}
        \hline
        &	0	&	1	&	-1	&	i	&	-i	&	\infty &	\T 	&\mathbb{R}&\mathbb{L} & (-\infty,0)\\
        \hline
        T(\cdot)	&	i	&	1	&	-1	&	\infty	&	0	&	-i	   &	\mathbb{R}	&\T &	(-\infty,0) &\mathbb{L}\\
        T^{-1}(\cdot)	&	-i	&	1	&	-1	&	0	&	\infty	&	i	   &	\mathbb{R}	&\T &	(-\infty,0) &\mathbb{L}\\
        \hline
          \end{array}
        \end{equation*}
        \caption{Values of the linear fractional transformations \eqref{eq:T} and \eqref{eq:T'}.
        Here $\mathbb{L}$ denotes the open arc of the unit circle $\T$ running counterclockwise from $i$ to $-i$.}
        \label{Table:T}
    \end{table}
    

    \begin{figure}
        \centering
        \begin{tikzpicture}[scale=1.35]
            \begin{scope}[xshift = -2.5cm]
                \draw[<->](-2,0)--(2,0);
                \draw[<->](0,-2)--(0,2);
                \draw[green,fill=green,opacity=.2](-1.9,0) rectangle (1.9,1.9);
                \draw[orange,fill=orange,opacity=.2](-1.9,0) rectangle (1.9,-1.9);
                \draw[blue,fill=blue,opacity=.2](1,0)arc(0:180:1cm);
                \draw[red,fill=red,opacity=.2](1,0)arc(0:-180:1cm);
                \draw [blue!50!green,very thick](0,0)circle(1cm);
                \draw[very thick,magenta,<->](-2,0)--(2,0);
                \draw[black,fill=black] (0,0) circle (.075cm);
                \draw[black,fill=blue](1,0) circle (.075cm);
                \draw[black,fill=purple] (-1,0)circle (.075cm);
                \draw[black,fill=cyan] (0,1) circle (.075cm);
                \draw[black,fill=orange] (0,-1) circle (.075cm);
                \draw[black,fill=green] (0,2.25)node[right]{\color{black}\, $\infty$} circle (.075cm);
                
                \node at (.1,1.2){$i$};
                \node at (.1,-1.2){$-i$};
                \node at (1.2,-.25){$1$};
                \node at (-1.3,-.25){$-1$};
            \end{scope}
            
            \node at (0,1.75) {$T(z)$};
            \begin{scope}[xshift = 2.5cm]
                \draw[<->](-2,0)--(2,0);
                \draw[<->](0,-2)--(0,2);
                \draw[green,fill=green,opacity=.2](-1,0)--(-1.9,0)--(-1.9,1.9)--(1.9,1.9)--(1.9,0)--(1,0)--(1,0)arc(0:180:1cm);
                \draw[green,fill=green,opacity=.2](-1,0)--(-1.9,0)--(-1.9,-1.9)--(1.9,-1.9)--(1.9,0)--(1,0)--(1,0)arc(0:-180:1cm);
                \draw[orange,fill=orange,opacity=.4](0,0)circle(1cm);
                \draw[red,fill=red,opacity=.2](1,0)arc(0:180:1cm);
                \draw[blue,fill=blue,opacity=.2](-1,0)--(-1.9,0)--(-1.9,1.9)--(1.9,1.9)--(1.9,0)--(1,0)--(1,0)arc(0:180:1cm);
                \draw [magenta, very thick](0,0)circle(1cm);
                \draw[very thick,blue!50!green,<->](-2,0)--(2,0);
                 \draw[black,fill=black] (0,1) circle (.075cm);
                \draw[black,fill=blue] (1,0) circle (.075cm);
                \draw[black,fill=purple] (-1,0) circle (.075cm);
                \draw[black,fill=cyan] (0,2.25)node[right]{\color{black}\,$\infty$} circle (.075cm);
                 \draw[black,fill=green] (0,-1)circle (.075cm);
                \draw[black,fill=orange] (0,0) circle (.075cm);
                
                \node at (.1,1.2){$i$};
                \node at (.1,-1.2){$-i$};
                \node at (1.2,-.25){$1$};
                \node at (-1.3,-.25){$-1$};
            \end{scope}
            
            \draw[thick,->](-1,1)to[out=45,in = 135] (1,1);
        \end{tikzpicture}
        \caption{Illustration of the linear fractional transformation \eqref{eq:T}.  The points $1$ and $-1$
        are fixed by $T$. The unit disk gets mapped to the upper half plane.}
        \label{Figure:T}
    \end{figure}
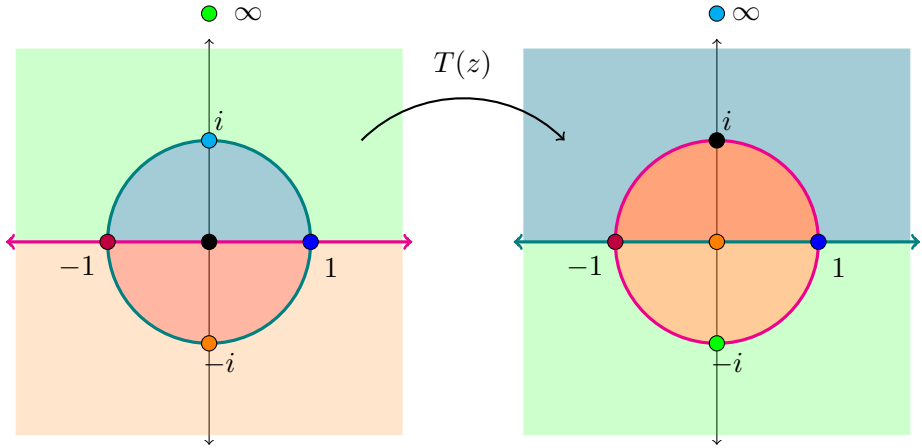
    
    If $\phi$ is inner, then $1 - i \phi$ and $1 + i \phi$ are both outer.
    Consequently, $T(\phi)$ belongs to $\RO$.  In fact,
    $T(\phi)$ belongs to $H^p$ for all $0 < p < 1$ (see \eqref{cvbnmdfdfi}).
    This is sharp, for $T(\phi)$ cannot belong to $H^1$ unless $\phi$ is constant (Theorem \ref{Theorem:p1}).
        Functions of the form $T(\phi)$, where $\phi$ is inner, form the basic building blocks from
    which all real outer functions can be built.  This prompts the following definition.
    
    \begin{Definition}\label{28734yiruhekfjd}
        A \emph{Cayley inner function} is a function of the form $T(\phi)$, where $\phi$ is an inner function.
    \end{Definition}

    The properties of Cayley inner functions are almost self evident. For one, a Cayley inner function is not inner at all, but outer!   
      Since a non-constant inner function $\phi$ maps $\D$ into $\D$ and the transformation $T$ maps $\D$ onto the upper half plane, it follows that
    $0 < \arg T(\phi) < \pi$ on $\D$, where $\arg$ refers to the principal branch of the argument.  
    Since $T(\phi)$ has real boundary values a.e.~on $\T$,
    the preceding implies that $\arg T(\phi)(\zeta) \in \{0,\pi\}$ for a.e.~$\zeta\in\T$.
    In fact, $\arg T(\phi)(\zeta) = \pi$ on $E = \phi^{-1}(\mathbb{L})$, where $\mathbb{L}$ denotes 
    the open arc of $\T$ that runs counterclockwise from $i$ to $-i$ (see Figure \ref{Figure:CT}).
    \begin{figure}
\centering
\begin{tikzpicture}
\begin{scope}[xshift=-5cm]
    \draw[blue,fill=blue,opacity=.2](0,0) circle (1.5cm);
    \draw[blue,line width=.1cm,opacity = .7] (0,1.5) arc (90:120:1.5);
    \draw[blue,line width=.1cm,opacity = .7,rotate=95] (0,1.5) arc (90:120:1.5);
        \draw[blue,line width=.1cm,opacity = .7,rotate=190] (0,1.5) arc (90:120:1.5);

    \node at (-.2,2) {$E$};
    \node at (-1.9,-.5) {$E$};
    \node at (.5,-1.7) {$E$};

\end{scope}
\draw[red,fill=red,opacity=.2](-1,0) circle (1.5cm);
\draw[blue,line width=.1cm,opacity=.7] (-1,1.5) arc (90:270:1.5);
\draw[<->](2,0)--(6,0);
\draw[<->](4,-2)--(4,2);
\draw[orange,fill=red,opacity=.2](2,0) rectangle (6,2);
\draw[blue,line width=.1cm,opacity=.7](2,0)--(4,0);

\draw[->] (-4,2)to[in=135,out=45](-2,2);
\node at (-3,2.7) {$\phi$};
\draw[->] (-0.2,2)to[in=135,out=45](1.8,2);
\node at (0.8,2.7) {$T$};

\draw[->] (-4,-2)to[in=-155,out=-25](2,-2);
\node at (-1,-3.2){$T(\phi)$};
\end{tikzpicture}
 \caption{Illustration of the mapping properties of a Cayley inner function.}
        \label{Figure:CT}
\end{figure}
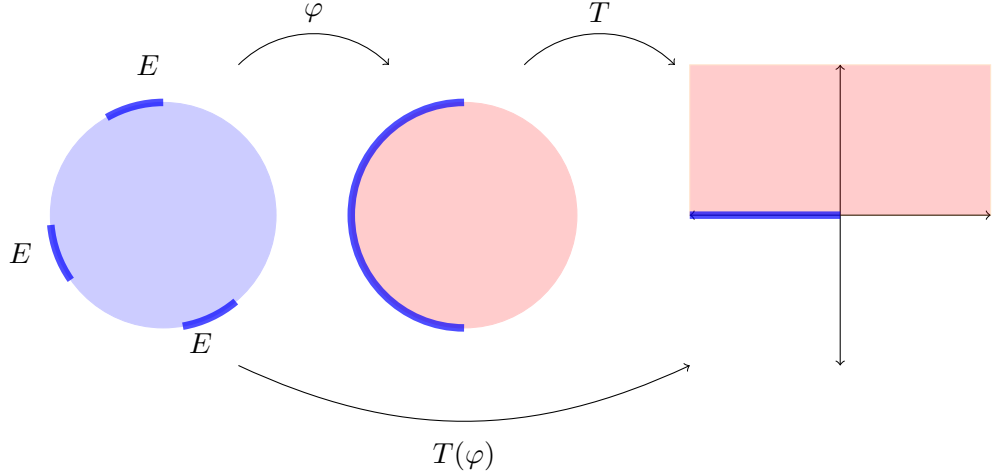
With a little work, the preceding computations can be reversed.

    For a Lebesgue measurable set $E \subseteq \T$ define 
    $$v_{E} = \mathscr{P}(\chi_{E}),$$
    the Poisson integral \eqref{eq:PIF} of the characteristic function $\chi_{E}$ for $E$. 
 Recall that $v_E$ is harmonic on $\D$
    and its (non-tangential) boundary values agree with $\chi_E$ a.e.~on $\T$. 
    Since $\chi_E$ only assumes the values $0$ and $1$, it follows that $v_E$
    assumes only values in $[0,1]$ on $\D$.  Furthermore, $v_E(0) = m(E)$. 
    
    If we normalize the harmonic conjugate $\widetilde{v_E}$ of $v_E$ so that $\widetilde{v_E}(0) = 0$, then 
    \begin{equation}\label{eq:CayleyInner}
        f_E = \exp[\pi(-\widetilde{v_E} + i v_E)]
    \end{equation}
    is analytic on $\D$, maps $\D$ into the upper half plane, and has real boundary values a.e.~on $\T$.  In particular, $f_E \in \mathsf{RO}$ and 
    $$\phi_{E} := T^{-1}(f_E)$$ is a bounded analytic function on $\D$ with unimodular values a.e.~on $\T$, i.e., an inner function. Thus $f_E = T(\phi_E)$ is a Cayley inner function and hence it belongs to $\mathsf{RO}$.
    
    The Cayley inner function $f_E$ obtained from \eqref{eq:CayleyInner}
    can alternatively be described as the exponential of a Herglotz integral, i.e., 
    \begin{equation}\label{eq:Herglotz}
        f_E(z) = \exp\left[ i\pi\int_E \frac{\zeta+z}{\zeta-z}\,dm(\zeta) \right].    
    \end{equation}
    Indeed, $v_E$ is the Poisson integral of its boundary function $\chi_E$ and thus 
    the integral in the exponential in \eqref{eq:Herglotz} can be obtained by analytic completion once 
    one recognizes 
    that $v_E$ is its real part and $\widetilde{v}_{E}(0) = 0$.
    
    \begin{Lemma}\label{Lemma:Cayley}
        Let $E \subseteq \T$ be Lebesgue measurable.  Then
        \begin{enumerate}
            \item $E = f_E^{-1}(-\infty,0) = \phi_E^{-1}(\mathbb{L})$, where $\mathbb{L} = \{ \zeta \in \mathbb{T} : \Re \zeta < 0 \}$;
            \item $\arg f_E = \pi \chi_E$ a.e.~on $\T$;
            \item $f_E(0) = e^{i\pi m(E) }$ and $\phi_E(0) = \tan[\frac{\pi}{2}(\frac{1}{2} - m(E))]$;
            \item $f_{\emptyset} \equiv 1$, $f_{\T} \equiv -1$, $\phi_{\emptyset} \equiv 1$, $\phi_{\T} \equiv -1$;
            \item If $f\in \mathsf{RO}$ satisfies (b), then $f= |f(0)|f_E$;
            \item If $\phi$ is inner and $\phi(0) \in \R$, then 
            	$\phi = \phi_E$, where $E = \phi^{-1}(\mathbb{L})$.
        \end{enumerate}
    \end{Lemma}
    
    \begin{proof}
        (a), (b) 
        The function $f_E$ from \eqref{eq:CayleyInner} is Cayley inner and
        satisfies $\arg f_E = \pi \chi_E$ a.e.~on $\T$ and $f_E^{-1}(-\infty,0) = E$ by construction.  The mapping properties of $T$
        ensure that $\phi_E =T^{-1}(f_E)$ satisfies $\phi_E^{-1}(\mathbb{L}) = E$ (Figure \ref{Figure:T}).  
        
        \noindent (c) Since $v_E(0) = m(E)$ and $\widetilde{v_E}(0) = 0$, we have 
        $$f_E(0) = e^{i \pi v_E(0)} = e^{i \pi m(E)}.$$
        A somewhat tedious, but ultimately elementary, calculation reveals that $$\phi_E(0) = \tan\left(\frac{\pi}{2}\left[\frac{1}{2} - m(E)\right]\right).$$
        
        \noindent (d) These identities come from the observations
        $$\widetilde{v_{\varnothing}} \equiv v_{\varnothing} \equiv 0, \quad v_{\T} \equiv 1, \quad \widetilde{v_{\T}} \equiv 0.$$
        
        \noindent (e) To prove this we first need a little detail. If $g \in N^{+}$ and outer then 
        $$g(z) = e^{i \gamma} \exp\left(\int_{\T} \frac{\zeta + z}{\zeta - z} \log |g(\zeta)| dm(\zeta)\right)$$
        and thus
        $$\log g(z) = i \gamma + \int_{\T} \frac{\zeta + z}{\zeta - z} \log |g(\zeta)| dm(\zeta).$$
        Any analytic function on $\D$ whose range is contained in $\{\Re z > 0\}$ belongs to $H^p$ for all $0 < p < 1$ \cite{Garnett}. One can show that the Cauchy transform 
        $$\int \frac{1}{1 - \overline{\zeta} z} d \mu(\zeta)$$ of a positive measure $\mu$ on $\T$ satisfies this property and thus belongs to $H^p$, $0 < p < 1$. Writing any complex measure as a linear combination of four positive measures shows that the Cauchy transform of any measure also belongs to $H^p$. From here it follows that 
        $$\int_{\T} \frac{\zeta + z}{\zeta - z} \log |g(\zeta)| dm(\zeta)$$ 
        belongs to $H^p$ and, as a result,  $\log g \in N^{+}$. 
        
      Apply this to the function $g = f/f_{E}$, where $f$ is $\mathsf{RO}$ and satisfies (b). Thus $\log g \in N^{+}$.
        Furthermore, 
        $$\left|\log \frac{f}{f_{E}}\right|^2 = \left(\log \frac{|f|}{|f_{E}|}\right)^2 + \left(\arg \frac{f}{f_{E}}\right)^2 = \left(\log \frac{|f|}{|f_{E}|}\right)^2 \quad \mbox{a.e.~on $\T$}.$$
        Thus, 
        $$\left|\log \frac{f}{f_{E}}\right| \leqslant \left|\log |f| + \log |f_E|\right| \in L^1.$$
        Hence $\log g \in N^{+}$ with integrable boundary values and thus, by Smirnov's Theorem (see \eqref{SmirnovThm}), $\log g \in H^1$. Since $f$ and $f_E$ share the same sign a.e.~on $\T$ we see that $\log f/f_E$ is real a.e.~on $\T$. By Theorem \ref{Theorem:p1} we conclude that $\log f/f_E$ is a constant function and thus $f$ and $f_E$ are positive scalar multiples of each other. Since $|f_{E}(0)| = 1$ it follows that $f = |f(0)| f_E$.
        
        \noindent (f) If $\phi$ is inner and $\phi(0) \in \R$ then $f = T(\phi)$ satisfies 
        $|f(0)| = 1$.  Let $E = f^{-1}(-\infty,0)$ and observe that (e)
        ensures that $f = f_E$.  Consequently, $\phi = T^{-1}(f_E) = \phi_E$, where $E = \phi^{-1}(\mathbb{L})$.
    \end{proof}
    
    A nice gem follows from the proof of this result: If $f$ and $g$ are outer functions with $\arg f = \arg g$ almost everywhere on $\T$, then $f = \lambda g$ where $\lambda > 0$. 

    \begin{Example}
        The functions $f_E$ enjoy some convenient multiplicative properties.  For example, since 
        $\chi_E + \chi_F = \chi_{E \cap F} + \chi_{E \cup F}$
        for any Lebesgue measurable sets $E,F \subseteq \T$, we can use \eqref{eq:Herglotz} to see that the corresponding Cayley
        inner functions satisfy
        \begin{equation*}
            f_E f_F = f_{E \cap F} f_{E \cup F}.
        \end{equation*}
        In particular, $f_E f_F = f_{E \cup F}$ whenever $E \cap F = \varnothing$.
        We also have 
        \begin{equation}\label{eq:fEE1}
            f_E f_{\T\backslash E} \equiv -1.  
        \end{equation}
        These identities can be extended to collections of three or more sets in an analogous way.
    \end{Example}

    \begin{Example}\label{Example:CayleyArc}
        Suppose that $\beta < \alpha < \beta + 2\pi$ and 
        \begin{equation*}
            E = \{ e^{i\theta} : \beta < \theta < \alpha \}
        \end{equation*}
        is an arc in $\T$, running counterclockwise from $e^{i\beta}$ to $e^{i\alpha}$.  Then 
        we can obtain $f_E$ from \eqref{eq:Herglotz}:
        \begin{equation*}
          f_E(z) = \exp\left[ \frac{i}{2} \int_{\beta}^{\alpha} \frac{e^{i\theta}+z}{e^{i\theta}-z}\,d\theta \right].
        \end{equation*}
        Some routine calculations show that
        \begin{equation}\label{eq:CayleyArc}
            f_E(z) = e^{-i (\frac{\alpha-\beta}{2})} \left( \frac{e^{i\alpha}-z}{e^{i\beta}-z} \right)
        \end{equation}
        and confirm that
        \begin{align*}
            \arg f_E(e^{i\theta})	
            &=	\arg\frac{ e^{i \alpha/2} - e^{-i \alpha/2} e^{i\theta}} { e^{i \beta/2} - e^{-i \beta/2} e^{i\theta}}	\\
            &=	\arg\frac{e^{i(\alpha-\theta)/2}-e^{-i(\alpha-\theta)/2}}  {e^{i(\beta-\theta)/2}-e^{-i(\beta-\theta)/2}}\\
            &=	\arg \left[\sin\Big(\frac{\theta-\alpha}{2}\Big) / \sin\Big(\frac{\theta-\beta}{2}\Big) \right]	\\
            &=	\pi\chi_E(e^{i\theta}),
        \end{align*}
        as expected. 
        
        Since $f_E$ is a linear fractional transformation, it follows that the inner function
        $\phi_E = T^{-1}(f_E)$ is also a linear fractional transformation.  This means it must be a unimodular
        constant multiple of a single Blaschke factor.  In what follows, it will be convenient to assume that 
        $0 < \alpha - \beta < \pi$.  This ensures that 
        \begin{equation}\label{eq:ExampleArcsOrigin}
            \phi_E(0) = \tan\big(\tfrac{1}{4}[\pi - (\alpha - \beta)]\big),
        \end{equation}
        and so $0 < \phi_E(0) <1$.  Consequently, 
        \begin{equation*}
        \phi_E(z) = \frac{ |z_E| }{ z_E}  \frac{z_E -z}{1 - \overline{z_E} z}
        \end{equation*}
        for some $z_E \in \D$.  From \eqref{eq:ExampleArcsOrigin} we know that       
        \begin{equation*}
            |z_E| = \tan\big(\tfrac{1}{4}[\pi - (\alpha - \beta)] \big).
        \end{equation*}
        By symmetry, one expects $z_E$ to lie on the line segment joining the origin to the midpoint 
        $e^{i\frac{1}{2}(\alpha+\beta)}$ of the arc $E$.  Since $f_E(z_E) = T(0) = i$, in which $f_E$ is given by 
        \eqref{eq:CayleyArc}, we solve for $z_E$ in the equation
        \begin{equation*}
            e^{-i (\frac{\alpha-\beta}{2})} \left( \frac{e^{i\alpha}-z_E}{e^{i\beta}-z_E} \right) = i
        \end{equation*}
        to obtain
        \begin{align*}
            z_E 	
            &=	\frac{e^{i \alpha/2} - i e^{i \beta/2}}{e^{-i \alpha/2} - i e^{-i \beta/2}}	\\
            &= 	e^{i(\frac{\alpha+\beta}{2})}  \frac{e^{-i \beta/2} - i e^{-i \alpha/2}} {e^{-i \alpha/2} - i e^{-i \beta/2}}	\\
            &=	e^{i(\frac{\alpha+\beta}{2})}  \frac{e^{-i \frac{\beta}{2}} - e^{i\frac{\pi}{2} -i \frac{\alpha}{2}}}
            	{e^{-i \frac{\alpha}{2}} + e^{-i\frac{\pi}{2} -i \frac{\beta}{2} }}	\\
            &=	e^{i(\frac{\alpha+\beta}{2})} \frac{e^{i(-\frac{\pi}{4} + \frac{\alpha}{4} + \frac{\beta}{4})}}
            	{-ie^{i(\frac{\pi}{4} + \frac{\alpha}{4} + \frac{\beta}{4})}} \cdot
            	\frac{e^{-i \frac{\beta}{2}} - e^{i\frac{\pi}{2} -i \frac{\alpha}{2}}}
            	 {e^{-i \frac{\alpha}{2}} + e^{-i\frac{\pi}{2} -i \frac{\beta}{2} }}\\
            &=	e^{i(\frac{\alpha+\beta}{2})} \frac{-1}{-i}
            	 \frac{e^{i\frac{1}{4}[\pi - (\alpha - \beta)]} - e^{-i\frac{1}{4}[\pi - (\alpha - \beta)]}}
              {e^{i\frac{1}{4}[\pi - (\alpha - \beta)]} + e^{-i\frac{1}{4}[\pi - (\alpha - \beta)]}}\\ 
            &=	e^{i(\frac{\alpha+\beta}{2})} \tan\big(\tfrac{1}{4}[\pi - (\alpha - \beta)]\big)
        \end{align*}
        as expected.  In particular, $\phi_E$ is the single Blaschke factor with zero 
        \begin{equation}\label{eq:CircularArcZero}
            z_E = e^{i\frac{1}{2}(\alpha+\beta)} \tan\left[\frac{1}{4}\left(\pi - (\alpha - \beta)\right)\right].
        \end{equation}
    \end{Example}
    
    \begin{Example}
        Let $E \subseteq \T$ be a Lebesgue measurable and suppose that 
        $m(E) \neq \frac{1}{2}$.  Then Lemma \ref{Lemma:Cayley} guarantees that $\phi_E(0) \neq 0$
        and hence we may write
        \begin{equation*}
            \phi_E(z) = \prod_{n \geq 1} \frac{|z_n|}{z_n} \frac{ z_n - z}{1-\overline{z_n}z}
            \exp\left(- \int_{\T} \frac{\zeta+z}{\zeta-z} \, d\mu(\zeta) \right),
        \end{equation*}
        where $\{z_n\}_{n \geq 1}$ is a Blaschke sequence
        and $\mu$ is a finite, non-negative, singular measure on $\T$.  
        Appealing again to Lemma \ref{Lemma:Cayley}, we find that
        \begin{equation}\label{eq:EfromInner}
            \tan\left[\frac{\pi}{2}\left(\frac{1}{2} - m(E)\right)\right]= e^{-\mu(\T)} \prod_{n \geq 1} |z_n|.
        \end{equation}
    \end{Example}

    One can verify that the linear fractional transformation $T$ defined by \eqref{eq:T}
    satisfies the following algebraic identites:
        \begin{align}
          T(z_1 z_2) 	
          &=	\frac{T(z_1)T(z_2) + T(z_1) + T(z_2) - 1}{1 + T(z_1) + T(z_2) - T(z_1)T(z_2)}, \nonumber\\[5pt]
          T( z_2 / z_1 ) 	
          &=	\frac{T(z_1)T(z_2) - T(z_1) + T(z_2) + 1}{T(z_1)T(z_2) + T(z_1) - T(z_2) + 1} ,\label{eq:T-Quotient}\\[5pt]
          T(z_1 + z_2) 	
          &=	\frac{3 T(z_1)T(z_2) + i T(z_1) + i T(z_2) + 1}{3i + T(z_1) + T(z_2) + i T(z_1)T(z_2)}. \label{eq:T-Addition}
        \end{align}
    These often lead to some curious identities involving Cayley inner functions. Here are two such examples. 

    \begin{Example}    
        Suppose that $f_1 = T(\phi_1)$ and $f_2 = T(\phi_2)$ for some inner functions $\phi_1$ and $\phi_2$.  
        Then $f_1+f_2$ is real valued a.e.~on $\R$ and maps $\D$ into the upper half plane.
        Consequently, there is an inner function $\phi$ so that $f_1 + f_2 = T(\phi)$.  
        Since $(T \circ T)(z) = 1/z$ by \eqref{eq:Tp1}, it follows that
        $T(f_1) = 1/\phi_1$, $T(f_2) = 1/ \phi_2$, and 
        $T(f_1 + f_2) =  1 / \phi$.  Then \eqref{eq:T-Addition} reveals that 
        \begin{equation}\label{eq-3Formula}
            \phi = \frac{3i \phi_1 \phi_2 + \phi_1 + \phi_2 + i}{3 + i \phi_1 + i \phi_2 + \phi_1 \phi_2}.
        \end{equation}
        Although this does not look like an inner function, it is.  The denominator 
        \begin{equation*}
        3 + i \phi_1 + i \phi_2 + \phi_1 \phi_2 = (1 + i\phi_1) + (1 + i \phi_2) + (1 + \phi_1 \phi_2)
        \end{equation*}
        is the sum of three outer functions, each of which 
        assume values in the right half-plane, so it is outer.
        Thus $\phi$ is the inner factor of the numerator $3i \phi_1 \phi_2 + \phi_1 + \phi_2 + i$.
    \end{Example}

    \begin{Example}
        A trivial consequence of Helson's Theorem (Theorem \ref{Theorem:Helson}) is that
        each $f \in \RR$ can be written as $f = T(\psi_2 / \psi_1)$, where $\psi_1$ and $\psi_2$ are relatively
        prime inner functions.  This fact, and little bit of algebra, show that every function in $\RR$
        is a simple algebraic function of two Cayley inner functions.  Indeed, if $f_1 = T(\psi_1)$ and $f_2 = T(\psi_2)$,
        then \eqref{eq:T-Quotient} reveals that
        \begin{align*}
          f	
          &=	\frac{T(\psi_1)T(\psi_2) - T(\psi_1) + T(\psi_2) + 1}{T(\psi_1)T(\psi_2) + T(\psi_1) - T(\psi_2) + 1}	\\
          &=	\frac{f_1 f_2 - f_1 + f_2 +1}{f_1 f_2 + f_1 - f_2 + 1}.
        \end{align*}
    \end{Example}

\section{Unilateral Products of Cayley Inner Functions}

    We now consider the convergence of products of the form
    \begin{equation}\label{eq:Unilateral}
      \prod_{n \geq 1} T(\phi_n) = \prod_{n \geq 1} \left(i \frac{1 - i\phi_n}{1+i \phi_n} \right), 
    \end{equation}
    where $\phi_n$ is a sequence of inner functions.  We refer to such products 
    as \emph{unilateral products} to distinguish them from the \emph{bilateral products}
    (i.e., analogous products with indices ranging from $-\infty$ to $\infty$)
    that will be considered later.  As we will see, a completely satisfactory theory of unilateral products
    can be developed.  In contrast, bilateral products pose a host problems, not all of which
    have been resolved.

\subsection{Bounded argument}
    Suppose that $f \in \RO$ has bounded argument.  It is instructive to consider this special case
    before considering the general setting.  The approach below is essentially due to 
    Poltoratski \cite{MR1889082}.  
    
    Since $f$ has bounded argument, we may write
    \begin{equation}\label{eq:OfBA}
        f = |f(0)| \exp[\pi(-\widetilde{v}+iv)],
    \end{equation}
    where $v$ is non-negative, integer valued (to make $f$ real valued almost everywhere on $\T$), and bounded above by some integer $N$.
    For each positive integer $n$, let
    \begin{equation*}
        E_n = \{ \zeta \in \T : v(\zeta) \geq n\}
    \end{equation*}
    and observe that 
    \begin{equation*}
        E_1 \supseteq E_2 \supseteq \cdots \supseteq E_N \supseteq E_{N+1} = \varnothing.
    \end{equation*}
    Then
    \begin{equation*}
        v = \sum_{1 \leq n \leq N} \chi_{E_n}
    \end{equation*}
    so that
    \begin{equation*}
        \exp[\pi(-\widetilde{v}+iv)] = \prod_{1 \leq n \leq N} \exp[\pi(-\widetilde{\chi}_{E_n} + i \chi_{E_n})] = \prod_{1 \leq n \leq N} f_{E_n}.
    \end{equation*}
    Returning to \eqref{eq:OfBA} and letting $\phi_n = \phi_{E_n}$, the preceding yields
    \begin{equation*}
      f = |f(0)| \prod_{1 \leq n \leq N} T(\phi_n) = |f(0)| \prod_{1 \leq n \leq N} i \left( \frac{1-i\phi_n}{1+i\phi_n}\right).
    \end{equation*}
    Combining this observation with Theorem \ref{Theorem:Factor} yields the following result.

    \begin{Theorem}\label{ufhpaiuref}
    Suppose $f \in \RR^{+}$ is factored as $f = K_{f} R_{f}$ as in Theorem \ref{Theorem:Factor}, where $K_f$ is a Koebe inner function and $R_{f} \in \mathsf{RO}$. If $\arg R_{f}$ is bounded, then there are inner functions
        $\phi_1,\phi_2,\ldots,\phi_N$ so that
        \begin{equation*}
            f = K_f \prod_{1 \leq n \leq N} T(\phi_n).
        \end{equation*}
        Moreover, the product belongs to $H^p$ whenever $f$ belongs to $H^p$.
    \end{Theorem}	    

\subsection{A convergence criterion}

It turns out that practical necessary and sufficient conditions exist for determining
when products of the form \eqref{eq:Unilateral} converge.  In fact, as we will see in a moment, any function $f \in \RO$ with semibounded argument can be expanding
in a product of the form \eqref{eq:Unilateral} that converges absolutely and locally
uniformly on $\D$.  We first require a basic lemma.

    \begin{Lemma}\label{Lemma:ProductPreserve}
    	Let $\{z_n\}_{n \geq 1}$ be a sequence in $\C \backslash\{i\}$.
    	Then $\prod_{n \geq 1} z_n$ converges absolutely
    	if and only if $\prod_{n \geq 1} T(z_n)$ converges absolutely.
    \end{Lemma}
    
    \begin{proof}
        If $\prod_{n \geq 1} z_n$ converges absolutely, then $z_n \to 1$ and $1+iz_n$ is bounded away from zero.  Since
        \begin{equation}\label{eq:PreserveProduct}
        	1 - T(z_n) = \frac{1-i}{1+iz_n}(1-z_n),
        \end{equation}
        the forward implication follows.  If the product $\prod_{n \geq 1} T(z_n)$ converges absolutely,
        then $T(z_n) \to 1$.  Since $T(1) = 1$, we conclude that $z_n \to 1$ and hence $1+iz_n$ is bounded away from zero.
        Appealing to \eqref{eq:PreserveProduct} yields the reverse implication.
    \end{proof}

    The following lemma reduces the consideration of products of Cayley inner functions
    to products of inner functions.  This is a significant reduction, since determining whether or not a product
    of inner functions converges is straightforward (see Lemma \ref{sdlfkjsodif} below).
    
    \begin{Lemma}\label{Lemma:UnilateralConvergence}
      Let $\phi_n$ be a sequence of inner functions satisfying $\phi_n(0) \not = 0$ and let $f_n = T(\phi_n)$.  The following are equivalent:
      \begin{enumerate}
        \item The product $\prod_{n \geq 1} \phi_n(0)$ converges absolutely;
        \item The product $\prod_{n \geq 1} \phi_n$ converges absolutely and locally uniformly on $\D$;
        \item The product $\prod_{n \geq 1} f_n(0)$ converges absolutely;
        \item The product $\prod_{n \geq 1} f_n$ converges absolutely  and locally uniformly on $\D$.
      \end{enumerate}
    \end{Lemma}
    
    \begin{proof}
        (a) $\Rightarrow$ (b) Suppose that $\prod_{n \geq 1} \phi_n(0)$ converges absolutely.  For each $n \in \mathbb{N}$,
        the Schwarz' Lemma yields
        \begin{equation*}
            |\phi_n(z) - \phi_n(0)| \leq |z| \left|1 - \overline{\phi_n(0)} \phi_n(z) \right|, \quad z \in \D.
        \end{equation*}
        The above inequality can be rewritten as
        \begin{equation*}
            |(1-\phi_n(0)) - (1 - \phi_n(z))| \leq |z| \left| (1-\phi_n(z)) + \phi_n(z)(1 - \overline{\phi_n(0)}) \right|,
        \end{equation*}
        which implies
        \begin{equation*}
            |1 - \phi_n(z)| - |1 - \phi_n(0)| \leq |z| \big( | 1 - \phi_n(z)| + |1 -\phi_n(0)| \big).
        \end{equation*}
        From here we deduce that 
        \begin{equation}\label{eq:InnerSchwarz}
            |1 - \phi_n(z)| \leq \left( \frac{1+|z|}{1-|z|} \right) |1 - \phi_n(0)|,
        \end{equation}
        and hence the absolute convergence of  $\prod_{n \geq 1} \phi_n(0)$ implies
        the absolute and locally uniform convergence of $\prod_{n \geq 1} \phi_n$ on $\D$.
        
        \noindent (b) $\Rightarrow$ (d) Suppose that product $\prod_{n \geq 1} \phi_n$ converges absolutely and locally uniformly on $\D$. 
        Since $f_n = T(\phi_n)$ and $f_n(0) \not = i$, Lemma \ref{Lemma:ProductPreserve} implies that $\prod_{n \geq 1} f_n$ converges absolutely on $\D$.
        The uniform continuity of $T$ on compact subsets of $\D$ ensures that this convergence is locally uniform.
        
        \noindent (d) $\Rightarrow$ (c) This implication is trivial.
        
        \noindent (c) $\Rightarrow$ (a) Suppose that $\prod_{n \geq 1} f_n(0)$ converges absolutely.  Since $f_n = T(\phi_n)$, we see from
        \eqref{eq:Tp1} that $(T \circ T \circ T)(f_n) = \phi_n$.  Three successive applications of Lemma \ref{Lemma:ProductPreserve} guarantee
        that $\prod_{n \geq 1} \phi_n(0)$ converges absolutely.
    \end{proof}
    
    The preceding theorem reduces the study of unilateral products \eqref{eq:Unilateral} of Cayley inner
    functions to the study of infinite products of inner functions.  This is well-understood territory.  Indeed, let
    \begin{equation}\label{eq:BOIF}
      \phi_n(z) = e^{i\theta_n} \left( z^{m_n}\prod_{k \geq 1} \frac{|z_k^{(n)}|}{z_k^{(n)}} \frac{ z_k^{(n)}-z}{1-\overline{z_k^{(n)}}z} \right)
      \exp\left(- \int_{\T} \frac{\zeta+z}{\zeta-z} \, d\mu_n(\zeta) \right)
    \end{equation}
    be a sequence of inner functions, where
    \begin{enumerate}
        \item  $\theta_n \in [- \pi, \pi)$;
        \item $m_n \in \mathbb{N}$;
        \item $z_1^{(n)}, z_2^{(n)},\ldots$ is a Blaschke sequence for each $n$;
        \item $\mu_n$ is a finite, non-negative, singular
        measure on $\T$.
    \end{enumerate}
    In what follows, we allow the possibility that some of the zero sequences 
    $$z_1^{(n)}, z_2^{(n)},\ldots$$ are finite.
    Since this does not affect our arguments in any significant way, except encumbering our notation, we proceed as if
    each such zero sequence is infinite.  The reader should have no difficulty in patching up the argument to handle the most
    general case.
    
    \begin{Lemma}\label{sdlfkjsodif}
      The product $\prod_{n \geq 1} \phi_n$ of the inner functions \eqref{eq:BOIF}
      converges absolutely and locally uniformly on $\D$ if and only if
      \begin{enumerate}
        \item $\sum_{n \geq 1} \theta_n$ converges;
        \item $\sum_{n \geq 1} m_n < \infty$;
        \item $\{ z_k^{(n)} : n,k \geq 1 \}$ forms a Blaschke sequence:  $\sum_{n,k \geq 1} (1 - |z_k^{(n)}|) < \infty$;
        \item $\sum_{n \geq 1} \mu_n(\T) < \infty$.
      \end{enumerate}
    \end{Lemma}
    
    \begin{proof}
        The necessity of (b) is self evident, so we may assume that (b) holds and 
        that $\phi_n(0) \neq 0$ for $n \in \mathbb{N}$.  Lemma \ref{Lemma:UnilateralConvergence}
        says that $\prod_{n \geq 1} \phi_n$ converges absolutely and locally uniformly if and only if
        \begin{equation*}
            \prod_{n \geq 1} \left( e^{-\mu_n(\T) + i \theta_n} \prodk |z_k^{(n)}| \right)
        \end{equation*}
        converges absolutely.  This occurs if and only if
        \begin{equation*}
            \sum_{n \geq 1}  \left( (\mu_n(\T) - i \theta_n) - \log \prod_{k \geq 1} |z_k^{(n)}| \right)
        \end{equation*}
        converges absolutely.  The series above converges absolutely if and only if each of the series 
        $\sum_{n \geq 1} \theta_n$ and $\sum_{n \geq 1} \mu_n(\T)$ converge absolutely, and if
        \begin{equation*}
            \sum_{n \geq 1} \left( -\log \prodk |z_k^{(n)}| \right) = - \sum_{n,k \geq 1} \log |z_k^{(n)}|
        \end{equation*}
       converge absolutely (that is, if (a) and (d) hold).  
        However, $$\sum_{n,k \geq 1} \log |z_k^{(n)}|$$ converges if and only if (c) holds.
    \end{proof}

    Suppose that $\phi_n$ is a sequence of non-constant inner functions and $f_n = T(\phi_n)$.
    If $E_n = f_n^{-1}(-\infty,0)$, then Lemma \ref{Lemma:Cayley} implies that
    $f_n = |f_n(0)|f_{E_n}$.  Lemma \ref{Lemma:UnilateralConvergence} asserts that $\prod_{n \geq 1} f_n$ converges
    absolutely and locally uniformly in $\D$ if and only if
    \begin{equation*}
      \prod_{n \geq 1} f_n(0) = \prod_{n \geq 1}\left( |f_n(0)| e^{\frac{1}{2}i m(E_n)} \right) = 
    \exp\left(\frac{1}{2} i \sum_{n \geq 1} m(E_n) \right)  \prod_{n \geq 1} |f_n(0)| 
    \end{equation*}
    does.  Consequently, when considering unilateral products of Cayley inner functions, it
    suffices to consider products of the form
    $$\prod_{n \geq 1} f_{E_n}, \quad E_n \subseteq \T.$$ The following theorem
    addresses the convergence of such products.

    \begin{Theorem}\label{Theorem:UniE_n}
        If $E_n \subseteq \T $, then the following are equivalent:
        \begin{enumerate}
            \item $\sum_{n \geq 1} m(E_n) < \infty$;
            \item $v = \pi \sum_{n \geq 1} \chi_{E_n} \in L^1$;
            \item $\prod_{n \geq 1} f_{E_n}$ converges absolutely and locally uniformly on $\D$.
        \end{enumerate}
        The product in (c) belongs to $\RR^{+}$ if and only if $v \in L \log L$.
    \end{Theorem}
    
    \begin{proof}
        The equivalence of (a) and (b) is immediate.
        Lemma \ref{Lemma:UnilateralConvergence} tells us that (c) is equivalent
        to the convergence of 
        \begin{equation*}
            \prod_{n \geq 1} f_n(0) = \exp\left(\frac{1}{2}i \sum_{n \geq 1} m(E_n) \right),
        \end{equation*}
        which is equivalent to (a).  If the product in (c) converges to $$f = \exp[\pi(-\widetilde{v}+iv)],$$
        then $v \geq 0$ so the conjugate $\widetilde{v}$ belongs to $L^1$ if and only 
        if $v \in L \log L$ \cite[Thm.~4.4]{Duren} (Zygmund's theorem on the integrability of the conjugate function).
    \end{proof}

    \begin{Example}\label{Example:L1LogL}
        Suppose that $v \in L^1$ only assumes values in $\pi \mathbb{Z}$ and write
        $v = v^+ - v^-$, where $v^+$ and $v^-$ are non-negative $L^1$ functions.  
        Then two applications of Theorem \ref{Theorem:UniE_n} produce
        an analytic function $f$ on $\D$ that is real-valued a.e.~on $\T$ (in the sense
        of non-tangential limiting values) and that satisfies $\arg f = v$.  
        If both $v^+$ and $v^-$ belong to $L \log L$, then $f$ belongs to $\RR^+$.
    \end{Example}
    
    \begin{Example}
        Suppose that $E_1,E_2,\ldots$ are Lebesgue measurable subsets of 
        $\T$ and that $v = \pi \sum_{n \geq 1} \chi_{E_n}$ belongs to
        $L^1$ but not $L \log L$.  Theorem \ref{Theorem:UniE_n} yields  that
        $f = \prod_{n \geq 1} f_{E_n}$ is well defined.  However, Zygmund's Theorem
        asserts that $\widetilde{v} \not \in L^1$ and so $f = \exp[\pi(-\widetilde{v}+iv)]  \not \in N$.
    \end{Example}
    
    \begin{Example}
        Recall from Example \ref{Example:CayleyArc} that if $E$
        is the circular arc in $\T$, running counterclockwise from $e^{i\beta}$ to $e^{i\alpha}$
        and $0 < m(E) < \frac{1}{2}$, then $\phi_E$ is the single Blaschke factor with zero
        $$z_E = e^{i\frac{1}{2}(\alpha+\beta)} \tan\left(\frac{\pi}{2}\left[\frac{1}{2} - m(E)\right]\right).$$
        This provides a bijection $E \mapsto z_E$ between circular arcs $E$
        with $0 < m(E) < \frac{1}{2}$ and $\D \backslash \{0\}$.
        Suppose that 
        \begin{equation}\label{eq:BlaschkeExample}
          \prod_{n \geq 1} \frac{|z_n|}{z_n} \frac{z_n - z}{1 - \overline{z_n}z}
        \end{equation}
        is a Blaschke product.  For each zero $z_n$ there is a unique
        circular arc $E_n$ so that $z_n = z_{E_n}$.  The $n$th factor in the product 
        \eqref{eq:BlaschkeExample} is precisely $\phi_{E_n}$.  The above discussion shows that the product \eqref{eq:BlaschkeExample} converges absolutely and locally
        uniformly on $\D$ if and only if $\sum_{n \geq 1} m(E_n) < \infty$.  
        Consequently, the summability condition $\sum_{n \geq 1} m(E_n) < \infty$ must be equivalent to the Blaschke condition
        $\sum_{n \geq 1} (1 - |z_n|) < \infty$.  We can demonstrate this equivalence directly.
        Indeed, since
        \begin{equation*}
          |z_{E_n}| = \phi_{E_n}(0) = \tan[ \tfrac{\pi}{2}( \tfrac{1}{2} - m(E_n))],
        \end{equation*}
        and
        \begin{equation}
          \lim_{x\rightarrow 0} \frac{1 - \tan(\tfrac{\pi}{2}(\tfrac{1}{2} - x))}{x} = \pi,
        \end{equation}
        the limit comparison test shows that the series
        \begin{equation}
        \sum_{n \geq 1} m(E_n) \quad \text{and} \quad 
          \sum_{n \geq 1}(1 - |z_n|) = \sum_{n \geq 1} (1 - \tan(\tfrac{1}{4}(\pi - m(E_n))))
        \end{equation}
        converge or diverge together.
    \end{Example}
		
    \begin{Example}\label{Example:OpenUnion}
        If $E$ is an open subset of $\T $, then $E$ decomposes uniquely as the countable 
        union of disjoint circular arcs $E_n$.  Let us assume that there are infinitely
        many arcs involved in this decomposition.  At most one can have measure
        greater than $\frac{1}{2}$, so we may assume that $0 < m(E_n) < \frac{1}{2}$
        for $n \in \mathbb{N}$.
        Since $m(E) = \sum_{n \geq 1} m(E_n)$,
        Theorem \ref{Theorem:UniE_n} ensures that
        \begin{equation*}
            f_E = \prod_{n \geq 1} f_{E_n}.
        \end{equation*}
        Since each $E_n$ is an arc, Example \ref{Example:CayleyArc} tells us that
        $f_{E_n}$ is a linear fractional transformation. Thus $f_E$ is an infinite product
        of linear fractional transformations.
    \end{Example}
    
    \begin{Example}
        Suppose that $E \subseteq \T$ is a fat Cantor set, i.e., a Cantor set with positive Lebesgue measure.
        Then $\T \backslash E$ is an open set of positive measure, so Example \ref{Example:OpenUnion}
        shows us that $f_{\T\backslash E}$ is a product of linear fractional transformations.
        Since \eqref{eq:fEE1} implies that $f_E = -1 / f_{\T \backslash E}$, we conclude that
        $f_E$ is a product of linear fractional transformations.
    \end{Example}

    \begin{Example}
        Consider the atomic inner function
        \begin{equation*}
            \phi_{\rho}(z) = \exp \left( \rho\frac{z+1}{z-1} \right) , \quad \rho > 0.
        \end{equation*}
        Since $\phi_{\rho}(0) = e^{-\rho} > 0$, Lemma \ref{Lemma:Cayley} shows 
        that 
        $$\phi_{\rho} = \phi_{E(\rho)}, \quad E(\rho) = \phi_{\rho}^{-1}(\mathbb{L}).$$
        A computation shows that $z\in \T $ belongs to $E(\rho)$ if and only if $\rho\frac{z+1}{z-1}$ lies in one of the
        imaginary intervals $(\frac{\pi i}{2}, \frac{3 \pi i}{2}) \pmod{2\pi i n}$, where $n \in \mathbb{Z}$.  Since the
        linear fractional transformation $(z+1)/(z-1)$ is
        self-inverse, we see that $z$ belongs to $E(\rho)$ whenever $z$ lies in one of the 
        circular arcs $I_n(\rho)$ connecting the points
        \begin{equation}
            \frac{(4n+3)\pi - 2\rho i}{(4n+3)\pi +2\rho i} \quad\text{and}\quad \frac{(4n+1)\pi - 2\rho i}{(4n+1)\pi +2\rho i}.
        \end{equation}
        For $n \geq 0$, the arcs $I_n(\rho )$ lie on the bottom half of $\T $ and shrink rapidly, approaching the point $1$ as 
        $n\to \infty$.  The arcs $I_n(\rho )$ for $n < 0$ are the complex conjugates of the arcs $I_{n+1}(\rho)$ and lie on the upper
        half of $\T $.  The total measure of the arcs $I_n(\rho )$ can be computed with Lemma \ref{Lemma:Cayley}:
        \begin{equation*}
            \sum_{n\in \Z} m(I_n(\rho))
            =	m(E(\rho)) = \frac{1}{2} - \frac{2}{\pi}\tan^{-1} (e^{-\rho}). 
        \end{equation*}
        For any $\rho > 0$, Theorem \ref{Theorem:UniE_n} says that
        \begin{equation*}
            f_{E(\rho)} =    \prod_{n\in\Z} f_{I_n(\rho)}.
        \end{equation*}
        where the product converges either as a bilateral (i.e., symmetric) product or
        as two separate unilateral products.  Moreover, each term of the product
        is a linear fractional transformation (Example \ref{Example:CayleyArc}).
    \end{Example}

\section{Herglotz $A$-integral representations}
    We wish continue our work from the previous section to obtain an infinite product representation of an $f \in \RO$. 
    Recall that this starts by writing 
    \begin{equation*}
        f = |f(0)| \exp[ \pi (- \widetilde{v} + iv)],
    \end{equation*}
    where $v$ is integer valued.  Unlike the cases considered in the preceding section,
    here $v$ may be unbounded (above and below).  To handle this situation, we require
    a generalization of the classical Herglotz formula
    \begin{equation}\label{eq:HergInt}
        h(z) =  i \int_{\T} \frac{\zeta+z}{\zeta -z} \Im h(\zeta) dm(\zeta), \quad z \in \D,
    \end{equation}
    which holds when $h \in H^1$.  This formula permits us to recover an analytic
    function $h$ on $\D$ from the boundary values of its imaginary part.  
    
    This section is devoted
    to obtaining a suitable generalization of the Herglotz formula.  We will resume our discussion
    of product representations of $\RO$ functions in Section \ref{Section:Bilateral}.
    
\subsection{The $A$-integral}    
    Unfortunately, as hinted in Example \ref{Example:L1LogL}, we are not 
    always lucky enough to have $h \in H^1$.  The fact that $f \in N^+$ only ensures
    that $\widetilde{v} \in L^1$ while its harmonic conjugate $v$ need not belong to $L^1$. Thus we cannot
    immediately recover $h = - \widetilde{v} + iv$ from its imaginary part.
    Consequently, we need to develop a suitable replacement for \eqref{eq:HergInt}.  This is where the 
    $A$-integral comes in.
    
    If $h:\T\to\C$ is Lebesgue measurable, define
        \begin{equation*}
        \lambda_h(t) = m(\{ |h| > t\}), \quad t > 0,
    \end{equation*}
    to be the {\em distribution function} for $h$.   We say that $h$ belongs to $L_0^{1,\infty}$ if
    \begin{equation}\label{eq:L01}
         \lambda_h(t) = o\left(\frac{1}{t} \right).
    \end{equation}
    A short exercise will show that $L^1 \subset L_0^{1,\infty}$. A classical theorem of Kolmogorov \cite[p.~131]{MR2500010} says that if $w \in L^1$ then its harmonic conjugate $\widetilde{w}$ belongs to $L^{1, \infty}_{0}$. 

    \begin{Definition}
        $h:\T\to\C$ is \emph{$A$-integrable} if it belongs to $L_0^{1,\infty}$ and
        \begin{equation*}
            	\lim_{A\to \infty} \int_{ \{|h| \leq A \} } h \, dm
        \end{equation*}
        exists.  This limit is called the \emph{$A$-integral of $h$ over $\T$} and is denoted by 
        \begin{equation*}
            (A)\!\int_{\T}h \,dm.
        \end{equation*}
    \end{Definition}

 The theory of $A$-integrals was developed by Denjoy,
    Titchmarsh \cite{Tit}, Kolmogorov, Ul\textprime{}yanov \cite{Ulyanov},
    and Aleksandrov \cite{Alek-1981}. 
    
    An analytic function on $\D$ is said to belong to the space $H_0^{1,\infty}$ if it belongs to $N^{+}$ and its
    boundary function is in $L_0^{1,\infty}$.  Aleksandrov showed that such a function is the Cauchy $A$-integral of its
    boundary function \cite{Alek-1981}.  That is,  if $h$ is in $H_0^{1,\infty}$, then
    \begin{equation*}
    	h(z) = (A)\!\! \int_{\T} \frac{h(\zeta)}{1 - \overline{\zeta} z} dm(\zeta), \quad z \in \D.
    \end{equation*}
     A detailed proof of Aleksandrov's theorem can be found in \cite{MR2215991}.
    In order to establish an infinite product expansion for real outer functions, we require
    a Herglotz integral analogue of Aleksandrov's theorem.

\subsection{Herglotz $A$-integral representations}

    The following theorem draws heavily on the work of Aleksandrov \cite{Alek-1981}.
    The proof given below can be found in \cite{MR2021044}.      
		
	\begin{Theorem}\label{Theorem:AlexHerglotz}
		Let $h = u+iv \in H_0^{1,\infty}$ and $u(0)=0$.  For $A > 0$ let
		\begin{equation*}
					v_A = 	\begin{cases}
							v	&	\text{if $|v| \leq A$},	\\
							A	&	\text{if $v > A$},		\\
							-A	& 	\text{if $v < -A$}.
						\end{cases}
		\end{equation*}
		Then
		\begin{equation}\label{eq:HerglotzA}
				h(z) = \lim_{A\to\infty} i \int_{\T} 
					\frac{\zeta+z}{\zeta-z} v_A(\zeta) dm(\zeta),
		\end{equation}
		where the convergence is uniform on compact subsets of $\D$.
	\end{Theorem}
	
    Suppose that $h = u +iv$ belongs to $H_0^{1,\infty}$ and $u(0)=0$.
    Without loss of generality, we may assume that $v(0) = 0$, that is, $h(0) = 0$.
    Indeed, a short argument shows that we may replace $v$ with $v - v(0)$
    in the definition of $v_A$.  The fact that the Herglotz integral of the constant
    function $v(0)$ is itself allows this reduction to go through.
    
    For $t > 0$ let
    \begin{equation*}
        \rho_h(t)	= t \lambda_h(t).	
    \end{equation*}
    Since $h \in H_0^{1,\infty}$, its boundary function belongs to $L_0^{1,\infty}$, so \eqref{eq:L01}
    ensures that $\rho_h \to 0$ as $t \to \infty$.  For $A > 0$, let
    \begin{equation}\label{eq:SigmaH}
        \sigma_h(A) = \displaystyle{\sup_{t \geq A} \rho_h(t)}
    \end{equation}
    and observe that $\sigma_h(A) \to 0$ (as $A \to \infty$) as well.
    
    The proof of Theorem \ref{Theorem:AlexHerglotz} requires the following three technical lemmas.

\begin{Lemma}\label{Lemma:HFS1}
    For $A > 0$ and $h \in H^{1, \infty}_{0}$,
    \begin{equation}\label{eq:HFSS1}
        \left| \displaystyle{\int_{|h|\leq A}} h\,dm   \right| \leq \rho_h(A) + 2 \sqrt{\sigma_h(0) \sigma_h(A)}.
    \end{equation}
\end{Lemma}

\begin{proof}
    Let $A > 0$ and let $g$ be the outer function that satisfies
    \begin{enumerate}
        \item $g(0) > 0$;
        \item $|g| = 1$ on $\{ |h| \leq A \} \subseteq \T$;
        \item $|g| = A/|h|$ on $\{ |h| > A \}\subseteq \T$.
    \end{enumerate}
    By construction, the analytic function $gh$ vanishes at the origin
    and satisfies $|gh| \leq A$ a.e.~on $\T$.
    Consequently,
    \begin{align*}
        0
        &=\int_{\T} gh \,dm = \int_{|h| \leq A} gh \,dm + \int_{|h| > A} gh \,dm  \\
        &= \int_{|h| \leq A} h \,dm  + \int_{|h| \leq A} (g-1)h \,dm + \int_{|h| > A} gh \,dm ,
    \end{align*}
    and so
    \begin{equation}\label{weifjkndfs10097}
        \int_{|h| \leq A} h\,dm 
        =  - \underbrace{\int_{|h| > A} gh \, dm }_{I_1(A)} + \underbrace{\int_{|h|\leq A} (1-g)h\,dm}_{I_2(A)} .
    \end{equation}
    The definition of $g$ implies that
    \begin{equation*}
        I_1(A) \leq \int_{|h|> A} |gh|\,dm = A \int_{|h|> A} \,dm = A\lambda_h(A) = \rho_h(A),
    \end{equation*}
    which yields the first term on the right-hand side of \eqref{eq:HFSS1}.

    To estimate $I_2(A)$ from \eqref{weifjkndfs10097}, we first use the Cauchy--Schwarz inequality:
    \begin{equation}\label{eq:I1CSMFS}
        |I_2(A)|^2 \leq \left( \int_{|h|\leq A} |h|^2 \, dm \right)\left( \int_{|h|\leq A} |1-g|^2 \, dm \right).
    \end{equation}
    By the distributional identity (see \cite[p. 50 - 51]{MR2215991}) and \eqref{eq:SigmaH} we obtain 
    \begin{equation}\label{eq:HCIBTSWES1}
        \int_{|h|\leq A} |h|^2 \, dm = 2 \int_0^A t\lambda_h(t)\, dt \leq 2 \sigma_h(0)A.
    \end{equation}
    This already provides one of the terms required for \eqref{eq:HFSS1}.
    The second integral in \eqref{eq:I1CSMFS} is more troublesome.
    Let $w = \log|g|$, so that
    \begin{equation*}
        g = \exp(w + i\widetilde{w}).
    \end{equation*}
    The definition of $g$ says that $w \equiv 0$ on the set $\{ |h| \leq A \}$ and so
    \begin{equation*}
        |1-g| = |1 - e^{i\widetilde{w}}| \leq |\widetilde{w}|.
    \end{equation*}
    Using the fact that the $L^2$ norm of $\widetilde{w}$ is dominated by that of $w$ (see \eqref{sdiyfwr9efhudjk}), we find that
    \begin{align*}
        \int_{|h|\leq A} |1-g|^2\,dm	
        &\leq	\int_{\T} |\widetilde{w}|^2\,dm				
        \leq	\int_{\T} |w|^2\,dm	\\
        &= \int_{\T} (\log |g|)^2 \,dm 
        =	\int_{|h| > A} \left( \log \frac{|h|}{A} \right)^2 \, dm	\\
        &=	- \int_A^{\infty} \left(\log \frac{t}{A} \right)^2 \, d\lambda_h(t).
    \end{align*}
    Integrating by parts leads to
    \begin{align}
        2\int_A^{\infty} \frac{\lambda_h(t) \log\frac{t}{A}}{t} \, dt
        &\leq 	2 \sigma_h(A) \int_A^{\infty} \frac{\log \frac{t}{A}}{t^2} \, dt	\nonumber \\
        &=	\frac{2 \sigma_h(A)}{A} \int_1^{\infty} \frac{ \log s}{s^2} \, ds	\nonumber \\
        &=	\frac{2 \sigma_h(A)}{A}.\label{eq:HCIBTSWES2}
    \end{align}
    Returning to \eqref{eq:I1CSMFS} with the bounds \eqref{eq:HCIBTSWES1}
    and \eqref{eq:HCIBTSWES2} we obtain
    \begin{equation*}
        |I_2(A)| \leq 2 \sqrt{\sigma_h(0) \sigma_h(A)}.
    \end{equation*}
    This completes the proof of the lemma.
\end{proof}

    \begin{Lemma}\label{Lemma:HFS2}
    For $h \in H^{1, \infty}_{0}$,
        \begin{equation*} 
            h(z) = \lim_{A\rightarrow \infty} \int_{|h|\leq A} \frac{h(\zeta)}{1 - \overline{\zeta}z} dm(\zeta)
        \end{equation*}
        uniformly on compact subsets of $\D$.
    \end{Lemma}
    
    \begin{proof}
        For $z \in \D$ and $\zeta \in \T$, let
        \begin{equation}\label{eq:Step2.hz}
            h_z(\zeta) = \frac{\zeta(h(\zeta)-h(z))}{\zeta - z}
        \end{equation}
        and observe that
        \begin{align*}
            \int_{|h|\leq A} h_z\,dm 
            &=  \int_{|h|\leq A} \frac{\zeta h(\zeta)}{\zeta - z} dm(\zeta)
              - \int_{|h|\leq A} \frac{\zeta h(z)}{\zeta - z} dm(\zeta) \\
            &=  \int_{|h|\leq A} \frac{h(\zeta)}{1 - \overline{\zeta} z} dm(\zeta)
              - \int_{|h|\leq A} \frac{h(z)}{1 - \overline{\zeta} z} dm(\zeta).
        \end{align*}
        Lemma \ref{Lemma:HFS1} tells us that
        \begin{equation}\label{eq:Step2.1}
        	\left| \int_{|h_z|\leq A} h_z \, dm \right| \leq 2 \sqrt{\sigma_{h_z}(0)\sigma_{h_z}(A)} + \rho_{h_z}(A).
        \end{equation}
        We claim that the right-hand side of the preceding tends to $0$
        uniformly on compact subsets of $\D$ as $A \to \infty$.
        
        For each $r \in (0,1)$ let
        \begin{equation*}
            M_r = \max_{|z| \leq r} |h(z)|.
        \end{equation*}
        For $|z| \leq r$ and $\zeta \in \T$, we read from \eqref{eq:Step2.hz} that
        \begin{equation}\label{eq:More72}
            |h_z(\zeta)| \leq \frac{|h(\zeta)| + |h(z)|}{1-|z|} \leq \frac{|h(\zeta)| + M_r}{1-r},
        \end{equation}
        which implies that
        \begin{equation}\label{eq:HFSAIFG1}
            \lambda_{h_z}(t) \leq \lambda_h\big( (1-r)t - M_r \big).
        \end{equation}
        If $t > 2M_r /(1-r)$, then
        \begin{equation*}
            t < \frac{2}{1-r} \big( (1-r)t - M_r \big).
        \end{equation*}
        Multiplying \eqref{eq:HFSAIFG1} by the preceding we obtain
        \begin{equation}\label{eq:HFSAIFG2}
            t\lambda_{h_z}(t) \leq \frac{2}{1-r} \big( (1-r)t - M_r \big) \lambda_h\big( (1-r)t - M_r \big).
        \end{equation}
        Therefore,
\begin{equation}\label{519}
\sigma_{h_z}(t) \leq \frac{2}{1-r} \sigma_{h}((1-r)t-M_r). 
        \end{equation}       
         Now suppose that
        \begin{equation}\label{eq:OMG2r}
            \frac{2M_r}{1-r} \leq A,
        \end{equation}
        so that the inequality \eqref{519} is valid for $t = A$.  Then
        \begin{equation*}
            \frac{(1-r)A}{2} \leq (1-r)A - M_r,
        \end{equation*}
        so that 
        \begin{equation*}
        \rho_{h_z}(A) \leq \sigma_{h_z}(A) \leq \frac{2}{1-r} 
        	\sigma_h\left( \frac{(1-r)A}{2}\right)
        \end{equation*}
        and
               \begin{equation*}
            \sigma_{h_z}(0) \leq \frac {2}{1-r} \max\{M_r,\sigma_h(M_r)\}.
        \end{equation*}
        The preceding two estimates show that the right-hand side of \eqref{eq:Step2.1}
        tends to zero uniformly on $|z| \leq r$ as $A \to \infty$.  This proves our claim.
        
        For $A$ that satisfy \eqref{eq:OMG2r}, we let
        \begin{equation*}
            A_r = (1-r)A - M_r.
        \end{equation*}
        We claim that the difference between 
        \begin{equation}\label{eq:D2Int}
        	\int_{|h_z| \leq A} h_z \,dm	\quad \text{and}\quad	\int_{|h|\leq A_r} h_z \, dm
        \end{equation}
        tends to $0$ uniformly on $|z| \leq r$ as $A \to \infty$.
        If $|z| \leq r$, then \eqref{eq:More72} shows that
        \begin{equation*}
            \{ |h| \leq A_r \} \subseteq \{ |h_z| \leq A \}.
        \end{equation*}
        Consequently, the difference between the two integrals in \eqref{eq:D2Int}
        is bounded in absolute value by 
        \begin{equation*}
            (1-r)^{-1}A \lambda_h(A_r),
        \end{equation*}
        which, in turn, is bounded by
        \begin{equation*}
            \left( \frac{A}{(1-r)A_r} \right) \rho_h(A_r).
        \end{equation*}
        Since $A/ A_r$ remains bounded as $A \to \infty$, the preceding
        tends to $0$, as desired.
        
        We showed that 
        \begin{equation*}
            \int_{|h| \leq A} h_z \,dm \to 0
        \end{equation*}
        uniformly on $|z| \leq r$ for each $r$ in $(0,1)$; this was our claim above.
        Now observe that the difference between $h(z)$ and
        \begin{equation*}
             \int_{|h| \leq A} \frac{h(z)}{1 - \overline{\zeta} z} dm(\zeta)
        \end{equation*}
        is bounded in absolute value by
        \begin{equation*}
            (1-|z|)^{-1}|h(z)|\lambda_h(A)
        \end{equation*}
        and hence tends to zero uniformly on $|z| \leq r$.
        This concludes the proof of the lemma.
    \end{proof}

    \begin{Lemma}\label{Lemma:HFS3}
        \begin{equation*}
            \lim_{A \to \infty} \int_{|h|\leq A} \frac{h(\zeta)}{1 - \zeta \overline{z}} dm(\zeta) = 0
        \end{equation*}
        uniformly on compact subsets of $\D$.
    \end{Lemma}

		\begin{proof}
			This result can be obtained by applying
			Lemma \ref{Lemma:HFS1} to the function
			$$\frac{h(\zeta)}{1 - \overline{z}\zeta}$$ and then arguing as in
			the proof of Lemma \ref{Lemma:HFS2}.  The details are largely identical.
		\end{proof}
		
We are now ready to conclude the proof of Theorem \ref{Theorem:AlexHerglotz}.
From Lemmas \ref{Lemma:HFS2} and \ref{Lemma:HFS3} we have
\begin{align}
    h(z)	
    &=	\lim_{A \to \infty} 
    \int_{|h|\leq A} \frac{h(\zeta) - \overline{h(\zeta)}}{1 - \overline{\zeta} z} dm(\zeta)	\nonumber \\
    &=	\lim_{A \to \infty} i \int_{|h| \leq A}\frac{2 v(\zeta)}{1 - \overline{\zeta} z} dm(\zeta) \label{eq:Piq}
\end{align}
uniformly on compact subsets of $\D$.  Since $v(0) = 0$, Lemma \ref{Lemma:HFS1} ensures that
\begin{equation*}
    \lim_{A\to \infty} i \int_{|h|\leq A} v(\zeta) dm(\zeta) = 0.
\end{equation*}
Subtract this from \eqref{eq:Piq} to obtain
\begin{equation*}
h(z) = \lim_{A\to \infty} i \int_{|h|\leq A}
    \left( \frac{\zeta+z}{\zeta-z} \right) v(\zeta) dm(\zeta) 
\end{equation*}
uniformly on compact subsets of $\D$.  In light of the fact that
\begin{equation*}
\{ |h| \leq A \} \subseteq \{|v| \leq A \},
\end{equation*}
we see that the difference between
\begin{equation*}
    \int_{|h|\leq A} \left( \frac{\zeta+z}{\zeta-z} \right) v(\zeta) dm(\zeta)
    \quad \text{and}	\quad
     \int_{|v|\leq A} \left( \frac{\zeta+z}{\zeta-z} \right) v(\zeta) dm(\zeta)
\end{equation*}
is bounded in absolute value by 
\begin{equation*}
    \left( \frac{1+|z|}{1-|z|} \right)A \lambda_h(A).
\end{equation*}
Finally, the difference between
\begin{equation*}
     \int_{|v|\leq A} \left( \frac{\zeta+z}{\zeta-z} \right) v(\zeta) dm(\zeta)
    \quad 	\text{and}	\quad
    \int_{\T} \left( \frac{\zeta+z}{\zeta-z} \right) v_A(\zeta) dm(\zeta)
\end{equation*}
is bounded in absolute value by 
\begin{equation*}
\left(\frac{1+|z|}{1-|z|} \right)A \lambda_v(A)
\end{equation*}
so 
\begin{equation*}
	h(z) = \lim_{A\to\infty} i \int_{\T} 
	\frac{\zeta+z}{\zeta-z} v_A(\zeta) dm(\zeta),
\end{equation*}
which is the desired result \eqref{eq:HerglotzA}.  This concludes the proof
of Theorem \ref{Theorem:AlexHerglotz}.\qed

\section{Bilateral products}\label{Section:Bilateral}

    From Theorem \ref{ufhpaiuref} we can write every $\mathsf{RO}$ function with bounded argument as an infinite product. Our aim in this section is to establish a similar factorization theorem for $\mathsf{RO}$ functions 
    with possibly unbounded argument.  
    Suppose that $f \in \RO$ and write
    \begin{equation*}
        f = |f(0)| \exp[ \pi (u + iv)],
    \end{equation*}
    where $u \in L^1$.  Although this is not enough to say that
    $v \in L^1$, now that we have Theorem \ref{Theorem:AlexHerglotz}
    at our disposal, this does not pose an insurmountable obstacle.

\subsection{Factorization of $\RO$ functions}

    \begin{Theorem}\label{Theorem:Bilateral}
        If $f \in \RO$, then there exist inner functions $\phi_n^+$ and $\phi_n^-$ so that
        \begin{equation}\label{eq:BLF}
          f = |f(0)|  \prod_{n \geqslant 1} \left(\frac{1 -i \phi_n^+}{1 + i \phi_n^+} \right) \left(\frac{1 + i \phi_n^-}{1 - i \phi_n^-} \right),
        \end{equation}
        where the product converges locally uniformly on $\D$.
    \end{Theorem}
    
    \begin{proof}
        Write $f = |f(0)| \exp[ \pi (u + iv)]$, where $u \in L^1$.  
        Without loss of generality, we may assume that $|f(0)| = 1$; that is, $u(0) = 0$.
        For each positive integer $n$, let
        \begin{equation*}
            E_n^+ = \{v \geqslant  n\}, \qquad E_n^- = \{v \leqslant -n\},
        \end{equation*}
        so that
        \begin{equation*}
            E_1^+ \supseteq E_2^+ \supseteq \cdots \quad \text{and} \quad
            E_1^- \supseteq E_2^- \supseteq \cdots.
        \end{equation*}
        Let
        \begin{equation*}
        	f_n^+ = f_{E_n^+}, \qquad \phi_n^+ = \phi_{E_n^+}, \qquad f_n^- = f_{E_n^-}, \quad \text{and}\quad\phi_n^- = \phi_{E_n^-},
        \end{equation*}
        where $\phi_{E_n}^{\pm}$ are the inner functions described in Subsection \ref{Subsection:Cayley}
        and $$f_{E_n^{\pm}} = T(\phi_{E_n^{\pm}})$$ are the corresponding Cayley inner functions (recall Definition \ref{28734yiruhekfjd}).

        Since $v \in L_0^{1,\infty}$, an application of Theorem \ref{Theorem:AlexHerglotz} to the function $h = u+iv$ 
        implies that for each $r \in (0,1)$ the harmonic extension of $h$ to $\D$ satisfies
        \begin{align*}
            h 
            &= \lim_{A\to\infty} i \int_{\T} \frac{\zeta+z}{\zeta-z}v_A(\zeta) dm(\zeta)\\
            &= \lim_{A\to\infty} \sum_{1 \leq n \leq |A|} ( -\widetilde{\chi}_{E_n^+} + i\chi_{E_n^+} +\widetilde{\chi}_{E_n^-} - i\chi_{E_n^-}) 
        \end{align*}
        uniformly on $|z| \leq r$.  That is to say, the series representation
        \begin{equation*}
            h = \sum_{n \geq 1} ( -\widetilde{\chi}_{E_n^+} + i\chi_{E_n^+} +\widetilde{\chi}_{E_n^-} - i\chi_{E_n^-})
        \end{equation*}
        is valid on $|z| \leq r$ and the convergence is uniform.
        Consequently,
        \begin{align}
            f 
            = \exp \pi h  & = \exp\left(\pi  \sum_{n \geq 1} ( -\widetilde{\chi}_{E_n^+} + i\chi_{E_n^+} +\widetilde{\chi}_{E_n^-} - i\chi_{E_n^-})  \right) \label{eq:NthPP}\\
            &= \prod_{n \geq 1} \frac{\exp[\pi(-\widetilde{\chi}_{E_n^+} + i\chi_{E_n^+} )]}{\exp[\pi(-\widetilde{\chi}_{E_n^-} + i\chi_{E_n^-} )]} \nonumber \\
            &= \prod_{n \geq 1}  \frac{f_n^{+}}{f_{E_n^-}} \nonumber
        \end{align}
        uniformly on $|z| \leq r$.
    \end{proof}

    \begin{Corollary}
        Suppose that $f = I_f F$ is a non-constant function in $\RR^+$,
        where $I_f$ is inner and $F$ is outer. Then there exist inner functions $\phi_n^+$ and $\phi_n^-$ so that
        \begin{equation}\label{Product}
          f = |f(0)| K(I_f) \prod_{n \geqslant 1} \left(\frac{1 -i \phi_n^+}{1 + i \phi_n^+} \right) \left(\frac{1 + i \phi_n^-}{1 - i \phi_n^-} \right),
        \end{equation}
        where the product converges uniformly on compact subsets of $\D$.  The factor $K(I_f)$ is to be ignored if $I_f$ is constant.  
        Moreover, if $f$ belongs to $\RR ^p$, then the infinite product belongs to $\RR ^p$.
    \end{Corollary}
    
    \begin{proof}
        This follows from Theorem \ref{Theorem:Factor} and \ref{Theorem:Bilateral}.
    \end{proof}

\subsection{Absolute Convergence}

    Theorem \ref{Theorem:Bilateral} does not guarantee absolute convergence and it is
    not clear whether absolute convergence occurs in general.  If
    \begin{equation}\label{eq:EqImpRI}
    	\sum_{n \geq 1} ( m(E_n^+) + m(E_n^-) ) < \infty,
    \end{equation}
    then $v \in L^1$ so Theorem \ref{Theorem:UniE_n} ensures that
    \begin{equation*}
        \prod_{n \geq 1} f_{E_n^+}\quad\text{and}\quad \prod_{n \geq 1} f_{E_n^-}
    \end{equation*}
    converge separately -- the convergence is absolute and locally uniform on $\D$.
    So in this case, the product \eqref{eq:BLF} converges absolutely.
    
    The $N$th partial product of \eqref{eq:NthPP} is
    \begin{equation*}
     \prod_{1 \leq n \leq N} \frac{\exp[\pi(-\widetilde{\chi}_{E_n^+} + i\chi_{E_n^+} )]}{\exp[\pi(-\widetilde{\chi}_{E_n^-} + i\chi_{E_n^-} )]} .
    \end{equation*}
    Since the harmonic conjugates $\widetilde{\chi}_{E_n^{\pm}}$ vanish at the origin, the value of the preceding
    product at $0$ is
    \begin{equation*}
        \exp\left[\pi i\left( m(E_n^+)- m(E_n^-) \right) \right].
    \end{equation*}
    Consequently, the product in \eqref{eq:BLF} converges absolutely at $0$ if and only if
    \begin{equation}\label{eq:ABSCC}
    	\sum_{n \geq 1} \left| m(E_n^+) - m(E_n^-) \right| < \infty.
    \end{equation}
    It is not clear if there is a function in $\RO$ for which \eqref{eq:ABSCC} fails.  This question
    was posed in \cite{MR2021044} and, frustratingly, remains open.  We hope that some spirited
    reader will someday be able to resolve this.
		
    This question can be viewed in terms of distribution functions.  
    Write $v = v_+ - v_-$, where $v_+$ and $v_-$ are non-negative.  Then
    $\lambda_{v_+}$ is equal to $m(E_{n+1}^+)$ on the interval $(n,n+1]$
    and $\lambda_{v_-}$ is equal to $m(E_{n+1}^-)$ on that same interval.  Consequently,
    \begin{equation*}
    	\text{\eqref{eq:EqImpRI} converges}\quad\iff\quad
	\text{$\displaystyle\int_0^{\infty} \left( \lambda_{v_+}(t) - \lambda_{v_-}(t) \right) \,dt$ converges},
    \end{equation*}
    where the integral above is regarded as an improper Riemann integral.  In other words, we have
    \begin{equation*}
    	\text{\eqref{eq:ABSCC} converges}\quad\iff\quad
	\big(\lambda_{v_+} - \lambda_{v_-} \big) \in L^1.
    \end{equation*}
    The next observation is from \cite{MR2021044}.  It follows by noting that 
    (i) for any real-valued measurable function $w$ on $\T $, the absolute integrability of
	$\lambda_{w_+} - \lambda_{w_-}$ on $[0,\infty)$ is equivalent to the same condition for any bounded perturbation of $w$;
	(ii) any $w$ as in (b) (see below) is a bounded perturbation of such an integer-valued $w$

		\begin{Theorem}\label{Theorem:AbsIntegrable}
			The following are equivalent:
			(a) For every function $f \in \RO$, the infinite product in the theorem converges absolutely at 0;
			(b) If $w$ is the conjugate of a real-valued function in $L^1$, then $\lambda_{w_+} - \lambda_{w_-}$
					is absolutely integrable on $[0,\infty)$.
		\end{Theorem}


\begin{Example}
In this example, we produce an $f \in \RO$ with nonintegrable argument, but such that the product in Theorem
\ref{Theorem:Bilateral} converges absolutely on $\D$.
Suppose that $f = \exp[\pi(u+iv)] \in \RO$ is such that 
\begin{enumerate}
    \item $v(\overline{\zeta}) = -v(\zeta)$ for a.e.~$\zeta \in \T$;
    \item $v$ is positive on the upper half of $\T$;
    \item $v$ is nonincreasing (with respect to $\theta$) on the upper half of $\T$.
\end{enumerate}
Conditions (a), (b), and (c) ensure that 
$E_n^+$ is an arc in the upper half of $\T$ with one endpoint at $1$ and 
$E_n^-$ is the reflection of $E_n^+$ across the real line.
Let the other endpoint be denoted $e^{i\alpha_n}$.  By \eqref{eq:CayleyArc}, we have 
\begin{equation*}
\frac{f_n^+(z)}{f_n^-(z)} = \frac{(z-e^{i\alpha_n})(z-e^{-i\alpha_n})}{(1-z)^2}.
\end{equation*}
Therefore,
\begin{align*}
\frac{f_n^+(z)}{f_n^-(z)} - 1 	
&=	 \frac{2z(1 - \cos\alpha_n)}{(1-z)^2}				\\
&=	\frac{4z \sin^2 \frac{\alpha_n}{2}}{(1-z)^2}			\\
&=	O\left( \frac{\alpha_n^2}{(1-z)^2} \right).
\end{align*}
Since $f \in \RO$, we know that $u \in L^1$ and $v \in L_0^{1,\infty}$, so \eqref{eq:L01}
tells us that $$\alpha_n = m(E_n) = o\left(\frac{1}{n}\right).$$  Thus,
\begin{equation*}
\sum_{n \geq 1} \left|\frac{f_n^+(z)}{f_n^-(z)} - 1 \right| < \infty,
\end{equation*}
with uniform convergence on compact subsets of $\D$.
This establishes the absolute convergence of the product.
\end{Example}

\subsection{A sufficient condition}
    Theorem \ref{Theorem:Bilateral} shows that any function in 
    $\RO$ enjoys a locally uniformly convergent bilateral product representation.  
    As we have seen above, this does not necessarily provide us with absolute convergence.    
    We investigate
    here a simple criterion which, given inner functions $\phi_n^+$ and $\phi_n^-$, imply the absolute 
    and locally uniform convergence of the 
    \emph{bilateral product}
    \begin{equation}\label{eq:BilateralProduct}
       \prod_{n \geq 1} \frac{T(\phi_n^+)}{T(\phi_n^-)}
    \end{equation}
    To this end, we require the following simple lemma.

    \begin{Lemma}
    	For $z \neq i$ and $w \neq -i$,
    	\begin{equation}\label{eq:MainIdentity}
    		1 - \frac{ T(z) }{T(w)}  
    		= \frac{2i(z-w)}{(1 + i z)(1 - i w)}.
    	\end{equation}
    \end{Lemma}
    
    \begin{proof}
    	This is a straightforward computation:
    	\begin{align*}
    		1 - \frac{ T(z) }{T(w)} 
    		&=	1 - \left(i \frac{1 - i z}{1 + i z} \right) \left(i \frac{1 - i w}{1 + i w} \right)^{-1}\\
    		& =	1 - \left( \frac{1 - i z}{1 + i z} \right) \left( \frac{1 + i w}{1 - i w} \right)\\
    		&=	\frac{(1 + i z)(1 - i w) - (1 - i z)(1 + i w)}{(1 + i z)(1 - i w)  } \\
    		&=	\frac{(1 + i z - i w + zw) - (1 - i z + i w + z + w)}{(1 + i z)(1 - i w)}\\
    		&=	\frac{2i(z - w)}{(1 + i z)(1 - i w)}.\qedhere
    	\end{align*}
    \end{proof}
    
    When considering bilateral products \eqref{eq:BilateralProduct}, it is natural to assume
    that for $r \in (0,1)$, the inner functions $\phi_n^+$ and $\phi_n^-$ are bounded away from $i$ and $-i$ as $n \to \infty$, respectively,
    on $|z| \leq r$.  This guarantees that $T(\phi_n^+)$ and $T(\phi_n^-)$ are bounded away from $\infty$ and $0$
    as $n \to \infty$, respectively, on $|z| \leq r$.

    \begin{Theorem}
        Suppose that $\phi_n^+$ and $\phi_n^-$ are two sequences of inner functions so that 
        for $r \in (0,1)$, the inner functions $\phi_n^+$ and $\phi_n^-$ are bounded away from $i$ and $-i$ on $|z| \leq r$ as $n \to \infty$, respectively.
         Then  \eqref{eq:BilateralProduct}
        converges absolutely and locally uniformly on $\D$
        if and only if $\sum_{n \geq 1} |\phi_n^+ - \phi_n^-|$
        converges locally uniformly on $\D$.
    \end{Theorem}

    \begin{proof}
        Fix $r \in (0,1)$ and suppose that
        \begin{equation*}
            0 < \delta < \sup_{|z| \leq r}|1 \pm i \phi_n^{\pm}|
        \end{equation*}
        for $|z| \leq r$.  Since $|1 \pm i \phi_n^{\pm}| \leq 2$ on $\D$, \eqref{eq:MainIdentity} tells us that
        \begin{equation*}
        	\frac{1}{2} \sum_{n \geq 1} | \phi_n^+ - \phi_n^-| 
        	\,\leq\, \sum_{n \geq 1} \left| 1 - \frac{ T(\phi_n^+) }{T(\phi_n^-)} \right|
        	\,\leq\, \frac{2}{\delta^2} \sum_{n \geq 1} | \phi_n^+ - \phi_n^-|. \qedhere
        \end{equation*}
    \end{proof}
    
    \begin{Example}
        Suppose that $\phi_n^+$ is a sequence of singular inner functions so that for $r \in (0,1)$,
        $\phi_n^+$ is bounded away from $\pm i$ on $|z| \leq r$ as $n \to \infty$.
        Now recall that the Blaschke products are uniformly dense in the set of all inner functions \cite[Cor.~6.5]{Garnett}.
        Let $\phi_n^-$ be a sequence of Blaschke products for which
        \begin{equation*}
            \sum_{n \geq 1} \norm{\phi_n^+ - \phi_n^-}_{\infty} < \infty.
        \end{equation*}
        Then the product \eqref{eq:BilateralProduct} converges.
    \end{Example}


\section{Real complex functions in operator theory}

We end this survey with a few connections the real complex functions make with operator theory. Since these vignettes are applications and not the main structure results of these real functions, as was the rest of the survey, we will be a bit skimpy on the details, referring the interested reader to the original sources in the literature. 
		
\subsection{Riesz projections for $0 < p < 1$}\label{Section:Riesz}

At first glance, the very title of this subsection refers to an absurdity.  Any serious analyst knows that the
Riesz projection operator 
$$\sum_{z \in \Z} \widehat{f}(n) \zeta^{n} \mapsto \sum_{n  \geq 0} \widehat{f}(n) \zeta^n$$
 cannot be properly defined for functions in $L^p$ when $0 < p < 1$. Indeed, one cannot even speak of Fourier series for such functions.  Bear with us.  

For $f \in L^1$ with Fourier series
\begin{equation*}
    f \sim \sum_{n \in \Z} \widehat{f}(n) \zeta^{n},
\end{equation*}
we may consider the ``analytic part'' 
$$\sum_{n \geq 0} \widehat{f}(n) \zeta^n$$ of this series.
When $1 < p < \infty$ and $f \in L^p$, the function
\begin{equation*}
    Pf = \sum_{n \geq 0} \widehat{f}(n) \zeta^n
\end{equation*}
belongs to $H^p$ and the linear transformation $f \mapsto Pf$ is a bounded projection from $L^p$ onto $H^p$ and is called the \emph{Riesz projection}.
In fact, \cite{MR1780482} tells us that
\begin{equation*}
    \norm{P}_{L^p \to H^p} =  \csc \left( \frac{\pi}{p} \right).
\end{equation*}
As mentioned earlier, this is no longer true when $p = 1$ or $p = \infty$. 

In this section we follow \cite{SCSSC} and show that an analog of the  Riesz projection
can be defined on $L^p$ when $0 < p < 1$ by working modulo the complexification of $\RR^p$.
In other words, the functions in $\RR^p$ are the culprit since their presence is the source of the
unboundedness of the Riesz projection.  Indeed, if $f \in \RR^p$, then 
$f = \overline{f}$ a.e.~on $\T$ and so the intersection $H^p \cap \overline{H^p}$ (in terms of boundary values on $\T$) contains
many non-constant functions.  This, it turns out, is the only obstruction.

    The set $\RR^+$ of all real Smirnov functions is a real subalgebra of $N^+$.
    It is natural to consider the complexification of $\RR^+$: 
    \begin{equation}\label{eq:CC}
        \CC^+ := \{ a + i b : a,b \in \RR^+\}.
    \end{equation}
    This is a complex subalgebra of $N^+$.  With respect to the translation invariant metric \eqref{eq:NMetric}
    it inherits from $N^+$ \cite{MR1761913},  $\CC^+$ is both a complete metric space and a topological algebra. 
    
    It is evident from the definition \eqref{eq:CC} that the set of boundary functions
    corresponding to the elements of $\CC^+$ is closed under complex conjugation.
    Indeed, if $a,b \in \RR^+$, then, in terms of boundary functions defined for a.e.~$\zeta \in \T$, 
    \begin{equation*}
        \overline{a(\zeta) + i b(\zeta)} = a(\zeta) - i b(\zeta) \in \CC^+.
    \end{equation*}
      Consequently, $\CC^+$ carries a canonical involution. Indeed,
    if $f = a+ib$, where $a,b \in \RR^+$, then we define
    \begin{equation*}
        \widetilde{f} = a - i b \in \CC^+.
    \end{equation*}
   The $\widetilde{\cdot}$ operation is a \emph{conjugation}
    on $\CC^+$: It is conjugate linear, involutive, and isometric.  Moreover, $\widetilde{\cdot}$
    preserves outer factors since
    \begin{equation}\label{eq:IsomCong}
    	|f|^2 = |a|^2 + |b|^2 = |\widetilde{f}|^2 \quad \mbox{a.e.~on $\T$}. 
    \end{equation}
     
    We now consider the intersection of the algebra $\CC^+$ with the Hardy spaces $H^p$.  For each $p$ let
    \begin{equation*}
        \CC^p := \CC^+ \cap H^p.
    \end{equation*}
    Theorem \ref{Theorem:p1} implies that $\CC^p$ contains 
    non-constant functions only when $p \in (0,1)$.
    It turns out that $\CC^p$ is the appropriate complex analogue of $\RR^p$ needed to define a ``Riesz projection'' on $L^p$ for $0 < p < 1$.  

    The metric on $H^p$ dominates the metric \eqref{eq:NMetric} on $N^+$, so we conclude that 
    $\CC^p$ is a closed subspace of $H^p$ for each $p \in (0,1)$.  Moreover, $\CC^p$ is closed
    under the conjugation $\widetilde{\cdot}$ since \eqref{eq:IsomCong} ensures that
    it is isometric on $H^p$.
    
    We leave it to the reader to verify the following theorem from  \cite{SCSSC}.

    \begin{Theorem} For $0 < p < 1$, 
        the following sets are identical.
        \begin{enumerate}
            \item $\CC^p$.
            \item $\RR^p + i \RR^p$.
            \item $H^p \cap \overline{H^p}$ (as boundary functions).
        \end{enumerate}
    \end{Theorem}

Let $0 < p < 1$. Since $C^p$ is a closed subspace of $L^p$, the quotient 
$L^p / C^p$ is an $F$-space under the standard quotient metric.  In other words, if $[f]$ denotes the equivalence class modulo $C^p$ of a function $f \in L^p$, then 
\[
\norm{[f]}^p := \inf_{\sigma\in C^p} \norm{f - \sigma}^{p}_{p}
\]
induces a translation invariant metric 
$$\rho([f],[g]) := \norm{[f]-[g]}^{p}_{p}$$ with respect to which $L^p / C^p$ is complete.  Similarly, we can
regard $H^p / C^p$ as a closed subspace of $L^p / C^p$ with respect to this metric.

A simple modification of a theorem of  Aleksandrov \cite{MR643380} says that we can decompose each $f \in L^p$, $0 < p < 1$, as 
$$f = u + v, \quad u \in H^p, v \in \overline{H^p},$$  
with some control over $\|u\|_{p}$ and $\|v\|_{p}$ \cite{SCSSC}. It is easily seen that this decomposition is unique modulo $C^p$ and hence
each equivalence class $[f] \in L^p/C^p$ decomposes uniquely as 
$$[f] = [u]+[v].$$  We can therefore define the \emph{Riesz projection} 
\begin{equation}\label{RPLp}
P:L^p/C^p \rightarrow H^p/C^p, \quad P[f] := [u].
\end{equation}  In a way, this map is an analogue of the Riesz projection operator
from $L^p$ to $H^p$ for $p \in (1,\infty)$.  
Indeed, if we regard equivalence classes as collections of boundary functions, then $P[f] \subseteq H^p$ and the Riesz projection returns the ``analytic part'' of $[f]$. The main theorem here is from \cite{SCSSC}.

\begin{Theorem}
   The Riesz projection from \eqref{RPLp} is bounded for each $p \in (0,1)$.  For every $f \in L^p$ we have ${\norm{P[f]}^p \leqslant K_p \norm{[f]}^p}$.
\end{Theorem}


\subsection{Kernels of Toeplitz operators}

 For each $\phi \in L^{\infty}$ one can define the Toeplitz operator \cite{MR2223704}
$$T_{\phi}: H^2 \to H^2, \quad T_{\phi} f = P(\phi f),$$ 
where $P: L^2 \to H^2$ is the Riesz projection. When $\phi \in H^{\infty}$, $T_{\phi} f = \phi f$ is a multiplication operator (a Laurent operator). The kernel $\ker T_{\phi}$ has been well studied \cite{MR1736197, MR776204, MR853630, MR1019277, MR1017680} and relates to the broad topic of ``nearly invariant'' subspaces of $H^2$. In particular there is the following theorem of Hayashi \cite{MR776204}.

\begin{Theorem}
If $\ker T_{\phi} \not = \{0\}$, then there is an outer function $F \in H^2$ such that $\ker T_{\phi} = \ker T_{\overline{z F}/F}$.
\end{Theorem}

 The connection to $\RR ^{+}$ is the following \cite{MR2186351}: 

\begin{Theorem}
If $F \in H^2$ is outer then 
$$\ker T_{\overline{z F}/F} = \{(a + i b) F: a, b \in \RR ^{+}\} \cap H^2.$$
\end{Theorem}

\subsection{A connection to pseudocontinuable functions}

A widely studied theorem of Beurling \cite{Duren, Garnett} says that the invariant subspaces of the shift operator 
$$S: H^2 \to H^2, \quad (S f)(z) = z f(z),$$
take the form $u H^2$ where $u$ is an inner function. Taking annihilators shows that the invariant subspaces of the backward shift operator 
$$S^{*}: H^2 \to H^2, \quad (S^{*} f)(z) = \frac{f(z) - f(0)}{z},$$
are of the form $(u H^2)^{\perp}$. Functions in $(u H^2)^{\perp}$ are often called the pseudo-continuable functions due to a theorem in \cite{MR0270196} (see also \cite{MR1895624}) which relates each $f \in (u H^2)^{\perp}$ with a meromorphic function on the exterior disk via matching radial boundary values.  Along the lines of our discussion of Toeplitz operators, we have the following result \cite{MR2186351}. 

\begin{Theorem}
For an inner function $u$ and $\zeta \in \T$ for which 
$$\lim_{r \to 1^{-}} u(r \zeta) = u(\zeta)$$
exists, define 
$$k_{\zeta}(z) = \frac{1 - \overline{u(\zeta)} u(z)}{1 - \overline{\zeta} z}.$$
Then we have 
$$(u H^2)^{\perp} = \{(a + i b) k_{\zeta}: a, b \in \RR ^{+}\} \cap H^2.$$
\end{Theorem}

We point out that similar results hold in $H^p$ when $1 \leq p < \infty$.

\subsection{A connection to unbounded Toeplitz operators}

For $\phi \in \RR ^{+}$ one can define the {\em unbounded} Toeplitz operator by first defining its domain 
$$\mathscr{D} = \{f \in H^2: \phi f \in H^2\}$$ and then defining the operator $T_{\phi}$ by 
$$T_{\phi}: \mathscr{D} \subseteq H^2 \to H^2, \quad T_{\phi} f = \phi f, \quad f \in \mathscr{D}.$$
Helson \cite{Helson} showed that the domain $\mathscr{D}$ of $T_{\phi}$ is dense in $H^2$: If $S f = z f$ is the shift operator on $H^2$ then $S \mathscr{D} \subseteq \mathscr{D}$ and so $S \mathscr{D}^{-} \subseteq \mathscr{D}^{-}$. By Beurling's classification of the $S$-invariant subspaces of $H^2$ we have 
$\mathscr{D}^{-} = I H^2$ for some inner function $I$. If $g \in H^2$ and $I g \in \mathscr{D}$, then 
$|\phi g|^2 = |\phi (g I)|^2$ and this last quantity is integrable since $\phi (g I) \in H^2$. This means that $g \in \mathscr{D}$ and so the common inner divisor $I$ of $\mathscr{D}$ is equal to one, making $\mathscr{D}^{-} = H^2$. 

In \cite{MR2418122} Sarason identified $\mathscr{D}$ as follows: One can always write $\phi$ as 
$$\phi = \frac{b}{a},$$
where $a, b \in H^{\infty}$, $a$ is outer, $a(0) > 0$, and $|a|^2 + |b|^2 = 1$ a.e.\! on $\T$.   Sarason  showed that 
$$\mathscr{D} = a H^2$$ and so, since $a$ is outer, one can see, via Beurling's theorem \cite[p. 114]{Duren}, that $\mathscr{D}$ is dense in $H^2$. This verifies what was shown by Helson above.

Since $\phi \in \RR ^{+}$, we have
$$\langle T_{\phi} f, g\rangle = \langle f, T_{\phi} g\rangle, \quad f, g \in \mathscr{D},$$
and thus $T_{\phi}$ is an unbounded symmetric operator on $H^2$. Furthermore, $T_{\phi}$ is also a closed operator. General theory of symmetric operators \cite{A-G} says that $T_{\phi} - w I$ has closed range for every $w \not \in \R$. Furthermore, if
$$\eta(w) = \mbox{dim}((\mbox{ran} (T_{\phi} - w I))^{\perp}),$$
where $\mbox{ran}$ denotes the range, 
then $\eta$ is constant on each of the half planes $\{\Im z > 0\}$ and $\{\Im z < 0\}$. The numbers $\eta(i)$ and $\eta(-i)$ are called the {\em deficiency indices} of $T_{\varphi}$. The following comes from Helson \cite{Helson}.

\begin{Theorem}
\hfill
\begin{enumerate}
\item If $\phi \in \RR ^{+}$ and $\eta(i)$ and $\eta(-i)$ are finite, then $\phi$ is a rational function. 
\item Given any pair $(m, n)$, where $m, n \in \mathbb{N} \cup \{\infty\}$, there is a $\phi \in \RR ^{+}$ such that $\eta(i) = m$ and $\eta(-i) = n$. 
\end{enumerate}
\end{Theorem}

Cowen \cite{Cowen} showed that two analytic Toeplitz operators $T_{\phi_1}, T_{\phi_2}$, where $\phi_1, \phi_2 \in H^{\infty}$, are unitarily equivalent if and only if $\phi_1 = \phi_2 \circ \psi$ for some automorphism $\psi$ of $\D$. A similar result was shown in \cite{MR3004956} when $\phi_1, \phi_1 \in \RR ^{+}$ and $T_{\phi_1}, T_{\phi_2}$ have deficiency indices $(1, 1)$. 

\subsection{Value distributions}

For $\phi \in \RR ^{+}$ the connection to unbounded Toeplitz operators (from the previous section) points out a useful connection to 
$$\mbox{card} \{z: \phi(z) = \beta\}, \quad \beta \not \in \R.$$ The unbounded symmetric Toeplitz operator $T_{\phi}$ has a densely defined adjoint $T_{\phi}^{*}$ and the Cauchy kernels 
$$k_{\lambda}(z) := \frac{1}{1 - \overline{\lambda} z}$$ belong to the domain of $T_{\phi}^{*}$. Furthermore, standard arguments show that 
$$T_{\phi}^{*} k_{\lambda} = \overline{\phi(\lambda)} k_{\lambda}.$$
Since 
$$(\mbox{ran} (T_{\phi}  - w I))^{\perp} = \ker({T_{\phi}^{*} - \overline{w} I}),$$
we see that 
$$\ker({T_{\phi}^{*} - \overline{w} I}) = \bigvee\{k_{\lambda}: \phi(\lambda) = w\}$$
and
$$\mbox{dim} (\ker({T_{\phi}^{*} - \overline{w} I})) = \mbox{card}\{\lambda: \phi(\lambda) = w\}.$$
Furthermore, from our earlier discussion of the deficiency indices of (unbounded) symmetric operators, the function 
$$w \mapsto \mbox{card}\{\lambda: \phi(\lambda) = w\}$$ is constant on each of the connected regions $\{\Im z > 0\}$ and $\{\Im z < 0\}$.

\appendix

\bibliography{RealSurvey}

\def\cprime{$'$} \def\cprime{$'$}
\providecommand{\bysame}{\leavevmode\hbox to3em{\hrulefill}\thinspace}
\providecommand{\MR}{\relax\ifhmode\unskip\space\fi MR }
\providecommand{\MRhref}[2]{%
  \href{http://www.ams.org/mathscinet-getitem?mr=#1}{#2}
}
\providecommand{\href}[2]{#2}
\begin{thebibliography}{10}

\bibitem{A-G}
N.~I. Akhiezer and I.~M. Glazman, \emph{Theory of linear operators in {H}ilbert
  space. {V}ol. {II}}, Translated from the Russian by Merlynd Nestell,
  Frederick Ungar Publishing Co., New York, 1963. \MR{MR0264421 (41 \#9015b)}

\bibitem{Alek-1981}
A.~B. Aleksandrov, \emph{${A}$-integrability of boundary values of harmonic
  functions}, Mat. Zametki \textbf{30} (1981), no.~1, 59--72, 154.
  \MR{83j:30039}

\bibitem{MR643380}
\bysame, \emph{Essays on nonlocally convex {H}ardy classes}, Complex analysis
  and spectral theory ({L}eningrad, 1979/1980), Lecture Notes in Math., vol.
  864, Springer, Berlin-New York, 1981, pp.~1--89. \MR{643380 (84h:46066)}

\bibitem{MR3004956}
Alexandru Aleman, R.~T.~W Martin, and William~T. Ross, \emph{On a theorem of
  {L}ivsic}, J. Funct. Anal. \textbf{264} (2013), no.~4, 999--1048.
  \MR{3004956}

\bibitem{MR2223704}
Albrecht B{\"o}ttcher and Bernd Silbermann, \emph{Analysis of {T}oeplitz
  operators}, second ed., Springer Monographs in Mathematics, Springer-Verlag,
  Berlin, 2006, Prepared jointly with Alexei Karlovich. \MR{2223704
  (2007k:47001)}

\bibitem{MR2215991}
Joseph~A. Cima, Alec~L. Matheson, and William~T. Ross, \emph{The {C}auchy
  transform}, Mathematical Surveys and Monographs, vol. 125, American
  Mathematical Society, Providence, RI, 2006. \MR{2215991 (2006m:30003)}

\bibitem{MR1761913}
Joseph~A. Cima and William~T. Ross, \emph{The backward shift on the {H}ardy
  space}, Mathematical Surveys and Monographs, vol.~79, American Mathematical
  Society, Providence, RI, 2000. \MR{1761913 (2002f:47068)}

\bibitem{Cowen}
C.~Cowen, \emph{On equivalence of {T}oeplitz operators}, J. Operator Theory
  \textbf{7} (1982), no.~1, 167--172. \MR{650201 (83d:47034)}

\bibitem{MR0270196}
R.~G. Douglas, H.~S. Shapiro, and A.~L. Shields, \emph{Cyclic vectors and
  invariant subspaces for the backward shift operator.}, Ann. Inst. Fourier
  (Grenoble) \textbf{20} (1970), no.~fasc. 1, 37--76. \MR{0270196 (42 \#5088)}

\bibitem{Duren}
P.~L. Duren, \emph{Theory of ${H}\sp{p}$ spaces}, Academic Press, New York,
  1970.

\bibitem{DurenUnivalent}
Peter~L. Duren, \emph{Univalent functions}, Grundlehren der Mathematischen
  Wissenschaften [Fundamental Principles of Mathematical Sciences], vol. 259,
  Springer-Verlag, New York, 1983. \MR{708494 (85j:30034)}

\bibitem{MR1736197}
Konstantin~M. Dyakonov, \emph{Kernels of {T}oeplitz operators via {B}ourgain's
  factorization theorem}, J. Funct. Anal. \textbf{170} (2000), no.~1, 93--106.
  \MR{1736197 (2000m:47036)}

\bibitem{SCSSC}
Stephan~Ramon Garcia, \emph{A {$\ast$}-closed subalgebra of the {S}mirnov
  class}, Proc. Amer. Math. Soc. \textbf{133} (2005), no.~7, 2051--2059
  (electronic). \MR{2137871 (2005m:30040)}

\bibitem{MR2186351}
\bysame, \emph{Conjugation, the backward shift, and {T}oeplitz kernels}, J.
  Operator Theory \textbf{54} (2005), no.~2, 239--250. \MR{2186351
  (2006g:30055)}

\bibitem{MR2021044}
Stephan~Ramon Garcia and Donald Sarason, \emph{Real outer functions}, Indiana
  Univ. Math. J. \textbf{52} (2003), no.~6, 1397--1412. \MR{2021044
  (2004k:30129)}

\bibitem{Garnett}
John~B. Garnett, \emph{Bounded analytic functions}, first ed., Graduate Texts
  in Mathematics, vol. 236, Springer, New York, 2007. \MR{2261424
  (2007e:30049)}

\bibitem{MR776204}
Eric Hayashi, \emph{The solution sets of extremal problems in {$H^1$}}, Proc.
  Amer. Math. Soc. \textbf{93} (1985), no.~4, 690--696. \MR{776204 (86e:30035)}

\bibitem{MR853630}
\bysame, \emph{The kernel of a {T}oeplitz operator}, Integral Equations
  Operator Theory \textbf{9} (1986), no.~4, 588--591. \MR{853630 (87m:47068)}

\bibitem{MR1019277}
\bysame, \emph{Classification of nearly invariant subspaces of the backward
  shift}, Proc. Amer. Math. Soc. \textbf{110} (1990), no.~2, 441--448.
  \MR{1019277 (90m:47013)}

\bibitem{Helson}
H.~Helson, \emph{Large analytic functions}, Linear operators in function spaces
  ({T}imi\c soara, 1988), Oper. Theory Adv. Appl., vol.~43, Birkh\"auser,
  Basel, 1990, pp.~209--216. \MR{1090128 (92c:30038)}

\bibitem{Helson2}
\bysame, \emph{Large analytic functions. {II}}, Analysis and partial
  differential equations, Lecture Notes in Pure and Appl. Math., vol. 122,
  Dekker, New York, 1990, pp.~217--220. \MR{1044789 (92c:30039)}

\bibitem{MR0236989}
H.~Helson and D.~Sarason, \emph{Past and future}, Math. Scand \textbf{21}
  (1967), 5--16 (1968). \MR{0236989 (38 \#5282)}

\bibitem{MR1780482}
Brian Hollenbeck and Igor~E. Verbitsky, \emph{Best constants for the {R}iesz
  projection}, J. Funct. Anal. \textbf{175} (2000), no.~2, 370--392.
  \MR{1780482 (2001i:42010)}

\bibitem{MR2500010}
Javad Mashreghi, \emph{Representation theorems in {H}ardy spaces}, London
  Mathematical Society Student Texts, vol.~74, Cambridge University Press,
  Cambridge, 2009. \MR{2500010 (2011e:30001)}

\bibitem{MR0213576}
J.~Neuwirth and D.~J. Newman, \emph{Positive {$H^{1/2}$} functions are
  constants}, Proc. Amer. Math. Soc. \textbf{18} (1967), 958. \MR{0213576 (35
  \#4436)}

\bibitem{MR1889082}
Alexei~G. Poltoratski, \emph{Properties of exposed points in the unit ball of
  {$H^1$}}, Indiana Univ. Math. J. \textbf{50} (2001), no.~4, 1789--1806.
  \MR{1889082 (2003a:30039)}

\bibitem{MR1895624}
William~T. Ross and Harold~S. Shapiro, \emph{Generalized analytic
  continuation}, University Lecture Series, vol.~25, American Mathematical
  Society, Providence, RI, 2002. \MR{1895624 (2003h:30003)}

\bibitem{MR1017680}
Donald Sarason, \emph{Nearly invariant subspaces of the backward shift},
  Contributions to operator theory and its applications ({M}esa, {AZ}, 1987),
  Oper. Theory Adv. Appl., vol.~35, Birkh\"auser, Basel, 1988, pp.~481--493.
  \MR{1017680 (90m:47012)}

\bibitem{MR2418122}
\bysame, \emph{Unbounded {T}oeplitz operators}, Integral Equations Operator
  Theory \textbf{61} (2008), no.~2, 281--298. \MR{2418122 (2010c:47073)}

\bibitem{Tit}
E.~Titchmarsh, \emph{On conjugate funtions}, Proc. London Math. Soc.
  \textbf{29} (1929), 49 -- 80.

\bibitem{Ulyanov}
P.~L. Ul{\cprime}yanov, \emph{On the ${A}$-{C}auchy integral. {I}}, Uspehi Mat.
  Nauk (N.S.) \textbf{11} (1956), no.~5(71), 223--229. \MR{18,726a}

\end{thebibliography}

\end{document}